\documentclass[11pt, reqno,english,empty]{amsart}
\usepackage{array}
\usepackage{latexsym, amssymb, enumerate, amsmath, array}
\usepackage{mathrsfs}
\usepackage{dsfont}
\usepackage{amsfonts,amsmath,amssymb,amsbsy}
\usepackage{graphicx}

\usepackage{latexsym}

\usepackage{enumerate}
\usepackage{mathrsfs}
\usepackage{stmaryrd}
\usepackage{amsopn}
\usepackage{amsmath}
\usepackage{amsfonts}
\usepackage{amsbsy}
\usepackage{amscd,indentfirst,epsfig}
\usepackage{hyperref}

\usepackage{amsthm,psfrag}
\usepackage{dsfont}
\usepackage{colordvi}
\usepackage{pstricks}

\usepackage{pifont}
\usepackage{enumerate}

\newtheorem{theo}{Theorem}[section]
\newtheorem{lemme}[theo]{Lemma}

\newtheorem{propo}[theo]{Proposition}

\newtheorem{hyp}[theo]{Assumption}
\newtheorem{nb}[theo]{Remark}
\newtheorem{defi}[theo]{Definition}
\newtheorem{exa}[theo]{Example}
\theoremstyle{definition}

\def \leq {\leqslant}
\def \geq {\geqslant}
\numberwithin{equation}{section}
\def\ind#1{\lower5pt\hbox{$\scriptstyle #1$}}
\def \le {\leqslant}
\def \ge {\geqslant}
\def \l {\ell}
\def \d {\, \mathrm{d} }

\def \Ss {\mathcal{S}}

\def \et {\mathbf{e_t}}
\def \ett {\mathbf{e_\tau}}
\def \es {\mathbf{e_s}}

\def \l {\ell}

\def \D {\mathscr{D}}

\def\Q {\mathcal{Q}}

\def\R{{\mathbb R}}
\def \S {{\mathbb S}^2}
\def \E {\mathcal{E}}
\def \n { {n}}
\def \gl {G_\ell}

\def \v {{v}}

\def \vb {\v_{\star}}

\def \It {\int_{\R^3 \times \S}}

\def \IR {\int_{\R^3}}
\def \IRR {\int_{\R^3 \times \R^3}}

\def \M {\mathcal{M}}

\def \T {\mathcal{T}}
\begin{document}

\title[Boltzmann equation for viscoelastic particles]{Boltzmann Model for viscoelastic particles:  asymptotic behavior, pointwise lower bounds and regularity} \author{R. Alonso \& B. Lods}
\address{\textbf{R. J. Alonso,} Dept. of Computational \& Applied Mathematics, Rice University
Houston, TX 77005-1892, USA., \texttt{Ricardo.J.Alonso@rice.edu} .}
\address{\textbf{B. Lods,} Universit\`{a} degli
Studi di Torino \& Collegio Carlo Alberto, Department of Economics and Statistics, Corso Unione Sovietica, 218/bis, 10134 Torino, \texttt{lods@econ.unito.it}.  }

\pagestyle{myheadings}  \maketitle

\hyphenation{bounda-ry rea-so-na-ble be-ha-vior pro-per-ties
cha-rac-te-ris-tic  coer-ci-vity}

\begin{abstract}
We investigate the long time behavior of a system of viscoelastic particles modeled with the homogeneous Boltzmann equation.  We prove the existence of a universal Maxwellian intermediate asymptotic state and explicit the rate of convergence towards it.  Exponential lower pointwise bounds and propagation of regularity are also studied.  These results can be seen as the generalization of several classical facts holding for the pseudo-Maxwellian and constant normal restitution models.
\end{abstract}
%\tableofcontents

\section{Introduction} \label{intro}
\setcounter{equation}{0}
We are interested here in the long-time behavior of the solution to the free-cooling Boltzmann equation for hard-spheres.  Namely, we consider the Cauchy problem
\begin{equation}
\label{be:force} \partial_\tau f(\tau,w)  = \Q_e(f,f)(\tau,w ), \qquad f(\tau=0,w)=f_0(w)
\end{equation}
where  the initial datum $f_0$ is a \textit{nonnegative} velocity function.  In such a description, the gas is described by the density of particles $f=f(\tau,w) \geq 0$ with velocity $w \in \mathbb{R}^3$ at time $\tau \geq 0$ while the collision operator $\Q_e$ models the interactions of particles by \textit{inelastic binary collisions}. A precise description of the Boltzmann collision operator $\Q_e$ will be given after we first describe the mechanical properties of particles interactions.
\subsection{Collision mechanism in granular gases}\label{Sec21} As well-known, the above equation is a well-accepted model that describes system composed by a large number granular particles which are assumed to be hard-spheres with equal mass (that we take to be $m=1$) and that undertake inelastic collisions.  The collision mechanism will be characterized by a single parameter, namely the \textit{coefficient of normal restitution} denoted by $e\in(0,1]$.  We are mostly interested in the case in which the coefficient of normal restitution depends solely on the impact velocity between particles. More precisely,  if $v$ and $\vb$  denote the velocities of two particles before collision, their respective velocities $v'$ and $\vb'$ after collision are such that
\begin{equation}\label{coef}
\big( u' \cdot \n \big)=- \big( u\cdot \n \big) \, e\big( |u \cdot \n| \big).
\end{equation}
The unitary vector $\n \in \mathbb{S}^2$  determines the impact direction, that is, $\n$ stands for the unit vector that points from the $v$-particle center to the $\vb$-particle center at the moment of impact.  Here above
\begin{equation}
u=v-\vb,\qquad u'=v'-\vb',
\end{equation}
denote respectively the relative velocity before and after collision.  The velocities after collision $v'$ and $\vb'$ are given, in virtue of \eqref{coef} and the conservation of momentum, by
\begin{equation}
\label{transfpre}
  v'=v-\frac{1+e}{2}\,\big( u\cdot \n \big)\n,
\qquad \vb'=\vb+\frac{1+e}{2}\, \big( u\cdot \n \big)\n.
\end{equation}
For the coefficient of normal restitution we shall adopt the following definition, see \cite{AlonsoIumj}
\begin{defi}\label{defiC} A coefficient of normal restitution $e(\cdot) \: :\: r \mapsto e(r) \in (0,1]$ belongs to the class $\mathcal{R}_0$ if it satisfies the following:
\begin{enumerate}
\item The mapping  $r \in \mathbb{R}^{+} \mapsto e(r) \in (0,1]$ is absolutely continuous and non-increasing.
\item The mapping $r\in\mathbb{R}^{+}\mapsto \vartheta_e(r):=r\;e(r )$ is strictly increasing.
\item $\lim_{r \to \infty} e(r)=e_0 \in [0,1)$.
\end{enumerate}
Moreover, for a given $\gamma \geq 0$, we shall say that $e(\cdot)$ belongs to the class $\mathcal{R}_\gamma$ if it belongs to $\mathcal{R}_0$ and
there exists $\mathfrak{a }>0$ such that
\begin{equation}\label{gamma}
e(r) \simeq 1-\mathfrak{a}r^\gamma \qquad \text{ as } \qquad r \simeq 0.\end{equation}
\end{defi}
\noindent
In the sequel we will \textit{always} assume that $e=e(\cdot)$ belong to the class $\mathcal{R}_\gamma$ with $\gamma \geq 0.$
\begin{exa} Here are some of the classical examples found in the literature:
\begin{enumerate}
\item An interesting first example is the Pseudo-Maxwellian model for dissipative hard spheres studied in \cite{CarGamba}.  This model is analogous to the classical Maxwell model for the elastic theory of gases, specifically, the kinetic part of the collision kernel is cleverly replaced by the square root of the system's temperature.  This model was investigated in \cite{CarGamba} assuming coefficient of normal restitution with similar structure to Definition \ref{defiC} but depending on the system's temperature rather than on the impact velocity.

\item Constant coefficient of normal restitution
\begin{equation*}
e(r)=e_0 \in (0,1]
\end{equation*}
which is clearly of class $\mathcal{R}_0$. This is the most documented example in the mathematical literature about Boltzmann equation for granular gases \cite{MiMo,MiMo2,MiMo3}.   The long time behavior of \eqref{be:force}  is well understood and the existence and uniqueness of a universal self-similar solution \emph{(homogeneous cooling state)} that attracts any solution to \eqref{be:force} has been proven in these references for $e_0$ sufficiently close to $1$.
\item A model of particular importance is the so-called viscoelastic hard spheres.  The coefficient of normal restitution is given by the expansion
\begin{equation}\label{visc}
e(r)=1+ \sum_{k=1}^\infty (-1)^k \, a_k \, r^{k/5}, \qquad r > 0
\end{equation}
where $a_k \geq 0$ for any $k \in  \mathbb{N}.$  We refer the reader to \cite{BrPo,PoSc} for the physical considerations leading to such expansion.  It is easy to check (see \cite[Appendix A]{AloLo1} for details) that $e(\cdot)$ given by \eqref{visc} belongs to the class $\mathcal{R}_0$.  Moreover, from \eqref{visc} one concludes that
\begin{equation*}
e(r) \simeq 1-a_1 r^{\frac{1}{5}} \qquad \text{ as } \qquad r \simeq 0.
\end{equation*}
Thus, such a coefficient of normal restitution is of class $\mathcal{R}_\gamma$ with $\gamma=1/5$.
\end{enumerate}
\end{exa}
\noindent
Due to the previous fundamental example \textit{(2)}, we will refer to any model with coefficient of normal restitution belonging to $\mathcal{R}_{\gamma}$ with $\gamma>0$  as a \textit{generalized viscoelastic particles} model.  We also remark here that additional assumptions on the function $e(\cdot)$ shall be needed later on for the propagation of regularity.  Note that using point  $\textit{(2)}$ in Definition \ref{defiC}, one concludes that the Jacobian of the transformation \eqref{transfpre} is given by
\begin{equation*}
J_e:=J_e\big(|u\cdot \n|\big)=\left|\dfrac{\partial (v',\vb')}{\partial (v,\vb)}\right|=|u\cdot \n| + |u\cdot \n|\dfrac{\d e}{\d r}(|u\cdot \n|)=\dfrac{\d \vartheta}{\d r}(|u\cdot \n|) >0.
\end{equation*}
Pre-collisional velocities $('v,'\vb)$ (resulting in $(v,\vb)$ after collision) can be therefore introduced through the relation
\begin{equation}\label{'v'vb}
v='\!\!v-\beta_e\big( |'\!\!u \cdot n| \big) \big( '\!\!u \cdot n \big) n\,, \qquad \vb='\!\!\vb+\beta_e\big( |'\!\!u \cdot n| \big)\big( '\!\!u \cdot n \big) n, \qquad '\!\!u='\!\!v-'\!\!\vb
\end{equation}
with $\beta_e(r)=\frac{1+e(r)}{2} \in \left(\frac{1}{2},1\right].$ In particular, the energy relation and the collision mechanism can be written as
\begin{equation}\label{energ}
|v|^{2}+| \vb|^{2}=|'\!v|^{2}+|'\!\vb|^{2}-\frac{1-e^{2}\big(|'\!u \cdot n|\big)}{2}\,\big('\!u \cdot n\big)^{2}, \qquad
u\cdot n=-e\big(|'\!u \cdot n|\big)\big('\!u \cdot n\big).
\end{equation}
Using the Definition \ref{defiC} one obtains the relation
\begin{equation} \label{'vxie}
'\!v =v-\xi_e\big(|u\cdot n|\big)n, \qquad \qquad '\vb=\vb + \xi_e\big(|u\cdot n|\big)n \end{equation}
with
$$\xi_e\big(|u\cdot n|\big)=\tfrac{1}{2}\big( \vartheta_e^{-1}(|u\cdot n|)+|u\cdot n| \big).$$
\subsection{Strong form of the Boltzmann operator}
For a given pair of distributions $f=f(v)$ and $g=g(v)$ and a given \textit{collision kernel} ${B}_0(u,\n)$, the Boltzmann collision operator is defined as the difference of two nonnegative operators (gain and loss operators respectively)
\begin{equation*}
\Q_{B_0,e}\big(f,g\big)=\Q_{B_0,e}^{+}\big(f,g\big)-\Q_{B_0,e}^{-}\big(f,g\big),
\end{equation*}
with
\begin{align}\label{Boltstrong}
\Q_{B_0,e}^{+}\big(f,g\big)(v)&=\It \dfrac{B_0(u,\n)}{e\big(|'u\cdot n|\big)J_e\big(|'u\cdot n|\big)}f('v)g('\vb)\d\vb\d\n\,,\nonumber\\
\Q_{B_0,e}^{-}\big(f,g\big)(v)&=f(v)\It  B_0(u,\n)g(\vb)\d\vb\d\n.
\end{align}
We will assume that the collision kernel $B_0(u,\n)$ is of the form
\begin{equation}\label{B0u}
B_0(u,\n)=\Phi\big(|u|\big)b_0\big(\widehat{u} \cdot \n\big)
\end{equation}
where $\Phi(\cdot)$ and $b_0(\cdot)$ are suitable nonnegative functions known as \textit{kinetic potential} and \textit{angular kernel} respectively.  For any fixed vector $\widehat{u}$, the angular kernel defines a measure on the sphere through the mapping $ \n \in \mathbb{S}^{2}\mapsto b_0\big( \widehat{u}\cdot n \big)\in[0,\infty]$ that it is assumed to satisfy the renormalized \textit{Grad's cut-off} assumption
\begin{equation}\label{normalization}
\left\| \big(\widehat{u}\cdot n \big)^{-1}\;b_{0}(\widehat{u}\cdot n) \right\|_{L^{1}(\mathbb{S}^{2},\d \n)}=2\pi \left\| s^{-1}\,b_{0}(s) \right \|_{L^{1}((-1,1),\d s)}=1.
\end{equation}
Of particular relevance is the hard spheres model which corresponds to the particular choice
\begin{equation*}
\Phi(|u|)=|u| \quad \text{and} \quad b_0(\widehat{u} \cdot \n)=\tfrac{1}{4\pi}|\widehat{u}\cdot \n|.
\end{equation*}
We shall often in the sequel consider the \textit{anisotropic hard spheres collision kernel} for which $\Phi(|u|)=|u|$ with angular kernel satisfying \eqref{normalization}.  For this particular model we simply denote the collision operator by $\Q_e$.
\subsection{Weak form of the Boltzmann operator}
It will be also convenient for our purposes to use the following equivalent form of the Boltzmann collision operator based on the so-called $\sigma$-parametrization of the post-collisional velocities.  More precisely, let $v$ and $\vb$ be a different particle velocities and let $\widehat{u}={u}/{|u|}$. Performing in \eqref{transfpre} the change of unknown
\begin{equation*}
\sigma=\widehat{u}-2 \,(\widehat{u}\cdot \n)\n \in \mathbb{S}^2
\end{equation*}
provides an alternative parametrization of the unit sphere $\mathbb{S}^2$. \footnote{The unit vector $\sigma$ points in the post-collisional relative velocity direction in the case of elastic collisions.}  In this case, the impact velocity reads
\begin{equation*}
|u\cdot\n|=|u| \,|\widehat{u} \cdot \n|=|u| \sqrt{\frac{1-\widehat{u} \cdot \sigma}{2}}.
\end{equation*}
Therefore,  the post-collisional velocities $(v',\vb')$ are given by
\begin{equation}\label{postsig}
v'=v-\beta_e\,\frac{u-|u|\sigma}{2}, \qquad \vb'=\vb+ \beta_e\,\frac{u-|u|\sigma}{2}
\end{equation}
where now
\begin{equation*}
\beta_e=\beta_e\left(|u| \sqrt{\tfrac{1-\widehat{u} \cdot \sigma}{2}}\right)=\frac{1+e}{2} \in \left(\tfrac{1}{2},1\right].
\end{equation*}
The following \textit{weak formulation} of the collision operator can be deduced from the $\sigma$-representation:  given a collision kernel $B(u,\sigma)$ one defines the associated  collision operator $\Q_{B,e}$  through the weak formulation
\begin{equation}\label{Ie3B}
\int_{\mathbb{R}^{3}}\Q_{B,e}(f,g)(v)\psi(v)\d v=\tfrac{1}{2}\int_{\mathbb{R}^{3} \times \mathbb{R}^3}f(v)g(\vb)\mathcal{A}_{B,e}[\psi](v,\vb)\;\d\vb\d v
\end{equation}
for any suitable test function $\psi=\psi(v)$.  Here
\begin{equation*}
\mathcal{A}_{B,e}[\psi](v,\vb)=\int_{\mathbb{S}^{2}}\Big(\psi( {v'})+\psi( {\vb'})-\psi(v)-\psi(\vb)\Big) {B}(u, \sigma)\d\sigma
\end{equation*}
with $v',\vb'$ are the post-collision velocities defined by \eqref{postsig} and the collision kernel $B(u,\sigma)$ is given by
\begin{equation}\label{Bu}
{B}(u,\sigma)=\Phi(|u|)b(\widehat{u} \cdot \sigma)
\end{equation}
where $\Phi(\cdot)$ is precisely the kinetic potential in \eqref{B0u} and the angular kernel $b(\cdot)$ is related to $b_0(\cdot)$ appearing in \eqref{B0u} by the relation
\begin{equation*}
b(\widehat{u}\cdot \sigma)= |\widehat{u} \cdot \n|^{-1}b_0\big(\widehat{u} \cdot \n\big).
\end{equation*}
In particular, the case of true-hard spheres corresponds to ${B}(u,\sigma)=\frac{1}{4\pi}|u|$.\footnote{In this setting the Grad's cut-off assumption simply reads as $\|b\|_{L^1(\mathbb{S}^2)}=1$.}
\subsection{Boltzmann model and main results}  The homogeneous Boltzmann equation for granular particles writes
\begin{equation}
\label{be:force} \partial_\tau f(\tau,w)  = \Q_e\big(f,f\big)(\tau,w ), \qquad f(\tau=0,w)=f_0(w)
\end{equation}
with given initial state $f_{0}$. We shall always assume that the initial datum $f_0\geq0$ satisfies
\begin{equation}\label{initial}
\IR f_0(w)\d w=1, \quad \IR f_0(w) w \d w =0 \quad \text{ and } \quad \IR f_0(w)|w|^3 \d w < \infty. \end{equation}
Under such hypothesis, the Cauchy problem \eqref{be:force} is well-posed: existence and uniqueness of a nonnegative solution $f(\tau,w)$ to \eqref{be:force} has been established in \cite{MMR} and the solution $f(\tau,w)$ satisfies \eqref{initial} for any $\tau > 0.$ The properties of this model depend heavily on the behavior at zero and infinity of the coefficient of normal restitution, in particular, they change depending whether $e(\cdot)$ belongs to the class $\mathcal{R}_0$ or the class $\mathcal{R}_\gamma$ with $\gamma > 0$.\\

\subsubsection{\textit{\textbf{Influence of slow moving particles}}.}  For generalized viscoelastic particles one sees from \eqref{gamma}, that $\lim_{r \to 0}e(r)=1$.  That is, particles with small relative velocities interact almost elastically.  This behavior for \emph{slow velocities} is responsible for the cooling law of the granular gas.  Denote
\begin{equation*}
\E(\tau)=\IR f(\tau,w)|w|^{2}\d w
\end{equation*}
the temperature of the solution $f(\tau)$ to \eqref{be:force}.  Under mild supplementary assumptions (that are fulfilled by both constant coefficient of normal restitution and viscoelastic hard spheres), the behavior of $\E(\tau)$ for large time is
\begin{equation*}
\E(\tau) \propto \left(1+\tau\right)^{-\frac{2}{1+\gamma}} \qquad  \gamma \geq 0.
\end{equation*}
The cooling rate of pseudo-Maxwellian interactions was treated in \cite{CarGamba} and the cooling rate for the model with constant coefficient of normal restitution ($\gamma=0$) was discussed in \cite{MiMo2}.  The generalization to viscoelastic models ($\gamma>0$) was extended later in \cite{AloLo1,AloLo2}.  This result shows that the elastic nature of the interactions for slow velocities fully determine the cooling law of the gas.  The difference of behavior between the cases $\gamma=0$ and $\gamma = 1/5$ was first noticed in \cite{BrPo} using formal arguments. A precise statement is the following, proved in \cite{AloLo1} (notice that such a result has been extended to initial datum having only finite initial entropy in \cite{AloLo2}):
\begin{theo}\label{entrHaff} Assume that the restitution coefficient $e(\cdot)$ is of class $\mathcal{R}_\gamma$ with $\gamma \geq 0.$  Let $f_0$ an initial distribution that satisfies \eqref{initial} and $f_0 \in L^p(\mathbb{R}^3)$ for some $1 < p < \infty$.  Let $f(\tau,w)$ be the unique solution to \eqref{be:force}. Then,
\begin{equation}\label{Haff's}
c_1 (1+\tau)^{-\frac{2}{1+\gamma}}  \leq \E(\tau) \leq  c_2 (1+\tau)^{-\frac{2}{1+\gamma}}, \qquad \tau \geq 0
\end{equation}
where $c_1$ and $c_2$ are positive constants depending only on $f_0$. More precisely, there exist  two positive constants $A_1, A_2 > 0$ such that
\begin{equation}\label{dEt}
A_1 \E^{\frac{3+\gamma}{2}}(\tau) \leq -\dfrac{1 }{2}\dfrac{\d}{\d \tau} \E(\tau) \leq A_2 \E^{\frac{3+\gamma}{2}}(\tau) \qquad \forall \tau \geq 0.
\end{equation}
\end{theo}
\medskip
\subsubsection{\textbf{\textit{Rescaled variables.}}} In order to study in a more accurate way the asymptotic behavior of $f(\tau,w)$, it is convenient to introduce rescaled  variables (see \cite{AloLo1,MiMo}). We shall assume here that $e(\cdot)$ belongs to the class $\mathcal{R}_\gamma$ with $\gamma \geq 0.$ For any solution $f=f(\tau,w)$ to \eqref{be:force}, we can associate the (self-similar) rescaled solution $g=g(t,v)$ by
\begin{equation}\label{resca}
f(\tau,w)=V(\tau)^3 g\big(t(\tau),V(\tau)w\big)
\end{equation}
where
\begin{align}\label{V(t)}
V(\tau)  =\sqrt{\dfrac{\E(0)}{\E(\tau)}} \qquad \text{ and } \qquad
t(\tau) =\int_0^\tau \dfrac{\d r}{V(r)}.
\end{align}
This scaling is important because it is precisely the one that stops the free cooling of the granular particles towards a Dirac measure with zero impulsion: indeed, under such a scaling the rescaled solution $g(t,v)$ is such that
$$\int_{\R^3} g(t,v)\,|v|^2\d v=\E(0)=:\E_0 \qquad \forall t > 0.$$
Notice that the mapping $\tau \in \R^+ \longmapsto t(\tau) \in \R^+$ is injective with inverse denoted by $s(\cdot)$. Then, as in \cite[Eq. (2.12)]{AloLo1}, one gets that the   rescaled solution satisfies
\begin{equation}\label{eqgt}
\partial_t g(t,v)+  {\xi(t)}
 \nabla_v \cdot \big( v g(t,v) \big) = \Q_{\et}\big(g,g\big)(t,v)\,,\quad g(t=0,v)=f_0(v)
\end{equation}
where
\begin{equation}\label{xitau}\xi(\cdot)= \dot{V}(s(\cdot)) \qquad \et(r)=e\left(z(t)r\right)\end{equation}
with $z(t)=\dfrac{1}{V(s(t))}.$ One can prove  without difficulty (see Section \ref{sec:preli} for details) that
$$\xi(t) \propto \left(1+\dfrac{\gamma}{1+\gamma} t\right)^{-1} \quad \text{ while} \qquad z(t) \propto \xi(t)^{1/\gamma} \quad \text{ for } t \to \infty.$$

Notice that when the coefficient of normal restitution is non constant the collision operator $\Q_{\et}$ is depending in time, thus, stationary solution to \eqref{eqgt} are not expected to happen. This is a major difference between the models with constant coefficient of normal restitution and the one associated to generalized viscoelastic particles.
\subsubsection{\textit{\textbf{Mawellian intermediate asymptotic state}}.} Let us turn back to our main concern: the long-time behavior of the free-cooling Boltzman equation \eqref{be:force}. The cases $\gamma=0$ and $\gamma>0$ enjoy a fundamental difference in the long time asymptotic behavior.  Indeed, the solution to the free-cooling Boltzmann equation \eqref{be:force} is known to converge towards a Dirac mass with zero impulsion
\begin{equation*}
f(\tau,\cdot) \rightharpoonup \delta_{0} \quad \text{ as } \quad \tau \to \infty
\end{equation*}
where the converge is meant in the space of probability measures on $\R^3$ endowed with the weak-topology \cite{CarGamba, MMR,AloLo1}. Therefore, it is expected that the solution $f(\tau,w)$ will converge first towards some intermediate asymptotic state $F(\tau,w)$ with $F(\tau,w) \to \delta_0$ as $\tau \to \infty$.
\begin{enumerate}
\item For constant coefficient of normal restitution $e(r)=\alpha \in (0,1)$ such state is given by a self-similar solution
\begin{equation*}
F_\alpha(\tau,w)=K(\tau)G_\alpha\big(V (\tau)w\big)
\end{equation*}
for some suitable scaling functions $K (\tau)$ and $V (\tau)$ independent of $\alpha$.  The profile $G_\alpha(\cdot)$ is precisely a steady state distribution of the rescaled Boltzmann equation \eqref{eqgt} (with $\gamma=0$ and $\xi(t)\equiv1$) and it is known as the \emph{homogeneous cooling state}.  The existence, exponential rate of convergence towards this state, uniqueness and stability in the weakly inelastic regime can be found in \cite{MiMo2} and \cite{MiMo3}. Notice however that $G_\alpha(\cdot)$ \textit{is not} a Mawellian distribution.
\item In the viscoelastic case the solution $f(\tau,w)$ is also converging towards an intermediate asymptotic state which is a self-similar solution to \eqref{be:force}.  The difference lies in the fact that such state is \textit{a time dependent \textit{Maxwellian} distribution}.   For the special case of pseudo-Maxwellian interactions, in this model the kinetic part of the collision kernel $|u|$ is replaced by $\sqrt{\mathcal{E}(\tau)}$, this fact was predicted in \cite[Section 6.2]{CarGamba} under conditions similar to those given in Definition \ref{defiC} but no convergence rate was provided.  See Section \ref{sec:discuss} for additional comments on the weakly inelastic regime.

\end{enumerate}
Let us state precisely the main result in the rescaled variables.
\begin{theo}\label{theo:Main}
Assume that $e(\cdot)$ belong to the class $\mathcal{R}_{\gamma}$ with $\gamma>0$ and that $r \in (0,\infty) \mapsto e(r)$ is infinitely differentiable with
\begin{equation*}
\sup_{r \geq 0}\, r \, e^{(k)}(r) < \infty\,,\quad\text{for any}\quad k \geq 1.
\end{equation*}
Let $f_0 \geq 0$ satisfy the  conditions given by \eqref{initial} and let $g(t,\cdot)$ be the solution to the rescaled equation \eqref{eqgt} with initial datum $g(0,\cdot)=f_0$. Moreover, assume that
$$f_0 \in \mathbb{H}^{m_0}_k \qquad \forall k \geq 0$$
for some sufficiently large (but explicit) $m_0 \geq 1.$ Then, for any $t_0 > 0$ and any $\varepsilon \in (0,1)$ the following holds
\begin{equation}\label{convergence}
\left\|g(t)-\mathcal{M}_0\right\|_{L^1} \leq C_1 \xi(t)^{\frac{1}{2(1+\varepsilon)}} \leq C_2\,\left(1+\tfrac{\gamma}{1+\gamma}\,t\right)^{-\frac{1}{2(1+\varepsilon)}}\,, \qquad \forall\; t \geq t_0
\end{equation}
for some positive constants $C_1,C_2$ depending on $\varepsilon$, $t_0$ and $m_0$. Here $\M_0$ denotes the Maxwellian distribution with same mass, momentum  and temperature than $f_0$, i.e.
\begin{equation}\label{maxwellianM0}
\M_0(v)=\frac{1}{(2\pi\E_0)^{3/2}}\exp\left(-\frac{|v|^2}{2\E_0}\right).
\end{equation}
\end{theo}
\noindent
We remark that the convergence towards this state is algebraic in time.  It is prescribed by \emph{the rate of convergence of the rescaled Boltzmann operator $\Q_{\et}^{+}$ towards the elastic Boltzmann operator $\Q_{1}^{+}$}.  One may wonder if such rate can be upgraded to exponential at least in some peculiar regime such as the weakly-inelastic regime introduced in \cite{AloLo2}. In the case of viscoelastic hard spheres $\gamma=\frac{1}{5}$ and  Theorem \ref{theo:Main} simply reads
\begin{equation*}
\left\|g(t)-\M_0\right\|_{L^1} \leq C\,\left(1+\tfrac{t}{6}\right)^{-\frac{1}{2(1+\varepsilon)}} \qquad \forall\;\varepsilon \in (0,1), \quad t > 0,
\end{equation*}
provided the initial datum $f_0$  is regular enough.

\subsubsection{\textit{\textbf{Influence of fast moving particles}}.} The influence of particles with large relative velocities on the granular gas dynamics is governed by the behavior of $e(r)$ at infinity.  In particular, in the case of constant coefficient of normal restitution, there is no distinction between slow and fast particles.  The influence of fast moving particles is quite subtle and can be quantified only in the tail behavior  of the solution $f$.  In particular, the (tail of the) lower pointwise bound that can be found for $f$ changes depending on the value $e_0$.  If $e_0 > 0$, the lower pointwise bound is, roughly speaking, the one established in the constant case in \cite{MiMo3}, that is, in this respect the gas behave as if the coefficient of normal restitution were constant $e(r)=e_0$.  Furthermore, if $e_0$ is sufficiently close to $1$ one expects to recover a pointwise lower bound which is ``almost'' Maxwellian.
\footnote{In the constant case assuming $e_0$ to be close to $1$ corresponds to the well-known weakly inelastic regime.  For $\gamma > 0$  this assumption implies that $\sup_{r\geq0}|e(r)-1|$ is small.  This is exactly the notion of weakly-inelastic regime introduced in \cite{AloLo2}.}  If $e_0=0$, then  fast particles will suffer almost \emph{sticky collisions} and the tail behavior then is prescribed by the behavior of these \emph{sticky particles} which yields the worst possible pointwise lower bound.  This discussion leads to our second theorem regarding the lower bound of the \textit{rescaled solution}  $g=g(t,v)$  to \eqref{eqgt}.
\begin{theo}\label{lowVisco0}
Assume that $e(\cdot)$ is a coefficient of normal restitution of class $\mathcal{R}_\gamma$ with $\gamma \geq0$ and with
\begin{equation*}
\liminf_{r \to \infty}\, e(r)=e_0 \in [0,1),
\end{equation*}
where if $e_0=0$, we also assume that there exists $m \in\mathbb{Z}$ and $C_m > 0$ such that
\begin{equation}\label{assKe}
\vartheta_e^{-1}(\varrho)+ \dfrac{\d}{\d \varrho}\vartheta_e^{-1}(\varrho)  \leq C_m \big(1+\varrho\big)^{m} \qquad \forall\; \varrho  >0.
\end{equation}
 Then,  for any $t_0 >0$ and any
\begin{equation}\label{def:a_0}
a_0 > \dfrac{2\log 2}{\log\left(1+\left(\frac{1+e_0}{2}\right)^2\right)}
\end{equation}
there exist positive $c_0,\, c_{1} >0$ such that the rescaled solution satisfies
\begin{equation*}
g(t,v) \geq c_0 \exp\left(-c_1|v|^{a_0}\right) \qquad \forall\; t \geq t_0,\quad v \in \mathbb{R}^3.
\end{equation*}
\end{theo}
\begin{nb} The assumption on $e(\cdot)$ provided by \eqref{assKe} is of technical nature and it is satisfied by all the models we have in mind, in particular, for viscoelastic particles for which $m=3/2$. We refer the reader to Section \ref{sec:low} for further details.
\end{nb}

\begin{nb}
One notices that the
worst exponent is the one corresponding to sticky collisions for which $e(r)=0$ for any $r$. In such a case,
\begin{equation*}
a_0 >\mathfrak{q}_0:=\frac{2\log 2}{\log(5/4)} \simeq 6.212.
\end{equation*}
Therefore, in any cases, there exist two constants $c_0, c_1 > 0$ such
that
\begin{equation*}
g(t,v) \geq c_0 \exp(-c_1|v|^{\mathfrak{q}_0}) \qquad \forall\, t \geq t_0,\quad v \in \mathbb{R}^3.
\end{equation*}
\end{nb}
\subsubsection{\textit{\textbf{Statement of the results in the original variables}}}
Using the notations of Theorem \ref{theo:Main} we have the main results written for the original variables.
\begin{theo}\label{theo:freecooling}
Assume that the coefficient of normal restitution $e(\cdot)$ satisfies the assumptions of Theorem \ref{theo:Main}. Let  $f_0 \geq 0$ satisfy the conditions given by \eqref{initial} and
\begin{equation*}
f_0 \in \mathbb{H}^{m_0}_k\,, \qquad \forall\; k \geq 0
\end{equation*}
for some explicit $m_0 \geq 1$.  Let $f(\tau,\cdot)$ denote the unique solution to \eqref{be:force} and let us introduce
\begin{equation*}
\overline{\M}_0(\tau,w)=V(\tau)^3\M_0(V(\tau)w)= \left(\frac{1}{2\pi\E(\tau)}\right)^{3/2}  \exp\left(-\frac{|w|^2}{2\E(\tau)}\right), \quad \forall \tau > 0\;,\; w \in \R^3.
\end{equation*}
Then, for any $\tau_0 > 0$ and any $\varepsilon > 0$, there exist  $A, B > 0$ such that
\begin{equation}\label{convft}
\|f(\tau,\cdot)-\overline{\M}_0(\tau,\cdot)\|_{L^1} \leq A\, \E(\tau)^{\frac{\gamma}{2(1+\epsilon)}} \leq B\,(1+\tau)^{-\frac{\gamma}{(1+\varepsilon)(1+\gamma)}} \qquad \forall \tau \geq \tau_0.\end{equation}
\end{theo}
\begin{proof}
The proof is a straightforward translation of Theorem \ref{theo:Main} with the use of the change of unknown \eqref{resca}.
\end{proof}
\begin{nb} For true viscoelastic hard spheres, as already observed $\gamma=1/5$ and \eqref{convft} reads
\begin{equation*}
\|f(\tau)-\overline{\M}_0(\tau)\|_{L^1} \leq C \left(1+\tau\right)^{-\frac{1}{12(1+\varepsilon)}} \qquad \forall\; \tau \geq \tau_0
\end{equation*}
for any $\tau_0 > 0$ and any $\varepsilon \in (0,1)$ (with $C$ depending on both $\varepsilon$ and $\tau_0$).
\end{nb}
\begin{nb} Notice that $\overline{\M}_0(\tau,\cdot)$ is converging to $\delta_0(\cdot)$ as $\tau \to \infty$ and it is playing the role of a Mawellian intermediate asymptotic state as described in Section 1.4.3.
\end{nb}
\begin{theo}
Under the assumptions of Theorem \ref{lowVisco0},  for any $\tau_0 >0$ and any
\begin{equation}\label{def:a_0}
a_0 > \dfrac{2\log 2}{\log\left(1+\left(\frac{1+e_0}{2}\right)^2\right)}
\end{equation}
there exist positive $c_0,\, c_{1} >0$ such that the solution $f(\tau,w)$ to \eqref{be:force} satisfies
\begin{equation*}
f(\tau,w) \geq c_0 \exp\left(-c_1\left(1+\tau\right)^{\frac{a_0}{1+\gamma}}|w|^{a_0}\right) \qquad \forall \tau \geq \tau_0,\;\:w \in \mathbb{R}^3.
\end{equation*}
\end{theo}
\begin{nb} As expected (recall that $f(\tau,\cdot) \to \delta_0(\cdot)$ as $\tau \to \infty$), the above pointwise lower bound degenerates as $\tau \to \infty$.
\end{nb}
\subsection{Notations} In the sequel we shall use the notation
$\langle \cdot\rangle = \sqrt{1+|\cdot|^2}$.  Additionally, for any
$p\in[1,+\infty)$, $\eta \in \R$ and weight function $\varpi\::\:\R^3 \to \R^+$, the weighted Lebesgue space is denoted by
\begin{equation*}
L^p_\eta (\varpi)= \left\{f: \R^3 \to \R \hbox{ measurable} \, ; \; \;
\| f \|_{L^p_\eta(\varpi)} := \left(\int_{\R^3} | f (v) |^p \, \langle v \rangle^{p\eta}\varpi(v)\d\v\right)^{1/p} < + \infty \right\}.
\end{equation*}
Similarly, we define the weighted Sobolev space $\mathbb{W}^{m,p}_\eta (\varpi)$, with $m \in \mathbb{N}$, using the norm
 \begin{equation*}\label{sobnorm}
\| f \|_{\mathbb{W}^{m,p}_\eta (\varpi)}  =   \left( \sum_{|s| \le m} \|\partial^s f\|_{L^p_\eta(\varpi)} ^p \right)^{1/p}.
\end{equation*}
In the particular case $p=2$ we denote $\mathbb{H}^m _\eta(\varpi)=\mathbb{W}^{m,2} _\eta(\varpi)$ and whenever $\varpi(v)\equiv 1$, we shall simply use $\mathbb{H}^m_\eta$.  This definition can be extended to $\mathbb{H}^s _\eta$ for any $s \ge 0$ by using Fourier transform.  The symbol $\partial^s$ denotes the partial derivative associated with the multi-index $s \in \mathbb{N}^3$: $\partial^s=\partial_{v_1}^{s_1}\partial_{v_2}^{s_2}\partial_{v_3}^{s_3}$. The order of the multi-index being defined as $|s|=s_1+s_2+s_3.$
\subsection{Method of proof and plan of the paper} The idea of proof for Theorem \ref{main} is surprisingly simple and it is based on the study of the \emph{relative entropy} of $g(t,v)$ with respect to the  Maxwellian $\M_0(v)$ with same mass momentum and energy  of $g(t,v)$ (recall that $g(t,v)$ is of \emph{constant} temperature):
\begin{equation*}
\mathcal{H}(t)=\mathcal{H}(g(t)|\M_0)=\IR g(t,v)\,\log\left(\dfrac{g(t,v)}{\M_0(v)}\right)\d v.
\end{equation*}
The convergence of $\mathcal{H}(t)$ towards zero is actually obtained from the following general result \cite[Theorem 4.1]{villcerc} provided $g(t,\cdot)$ is regular enough and satisfies a uniform lower bound such as the one provided in Theorem \ref{lowVisco0}:
\begin{theo}[\textbf{Villani}]\label{theo:Vill} For a given function $f \in L^1_2(\R^3)$, let $\M_f$ denote the Maxwellian
function with the same mass, momentum and energy as $f$. Assume that there exist $K_0 > 0$, $A_0 > 0$ and $q_0 \geq 2$ such that
\begin{equation}\label{villpoint}
f (v) \geq  K_0\,\exp\left(-A_0\,|v|^{q_0}\right) \qquad \forall v \in \R^3.
\end{equation}
Then, for all $\varepsilon \in (0,1)$, there exists a constant $\lambda_\varepsilon(f)$, depending on  $\varepsilon > 0$ and on $f$ only through its mass and energy and upper bounds on $A_0,$ $1/K_0$, $\|f\|_{L^1_{s}}$ and $\|f\|_{\mathbb{H}^m}$, where $s = s(\varepsilon, q_0) > 0$ and
$m = m(\varepsilon, q_0) > 0$, such that
\begin{equation*}
\D_1(f) \geq \lambda_\varepsilon(f) \left(\IR f(v)\log \left(\frac{f(v)}{\M_f(v)}\right)\d v \right)^{1+\varepsilon}
\end{equation*}
where $\D_1(f)$ is the  entropy dissipation functional associated to the elastic Boltzmann operator
\begin{equation*}
\D_1(f)=-\IR \Q_{1} \big(f,f\big)(v)\log\left(\frac{f(v)}{\M_f(v)}\right)\d v.
\end{equation*}
\end{theo}
\noindent
The use of Theorem \ref{theo:Vill} requires two main technical steps to complete the proof of Theorem \ref{main}.  First, the ``constant'' $\lambda_\varepsilon(f)$ in Theorem \ref{theo:Vill} is depending on higher Sobolev norms $\|f\|_{\mathbb{H}^m}$, thus, we need to prove the propagation of regularity for the rescaled solution to equation \eqref{eqgt}.  This is the object of Section \ref{sec:reg} whose main result reads.

\begin{theo}\label{theo:regu}
Let $f_0 \geq 0$ satisfy the  conditions given by \eqref{initial} and let $g(\tau,\cdot)$ be the solution to the rescaled equation \eqref{eqgt} with initial datum $g(0,\cdot)=f_0$. Then, for any $m \geq 0$ and $k \geq 0$, there exists some explicit $\eta_0=\eta_0(k,m)$ and $\eta_1=\eta_1(k,m)$ such that, if $f_0 \in \mathbb{H}^{m}_{k+\eta_0} \cap L^1_{k+\eta_1}$ then,
\begin{equation*}
\sup_{t \geq 0}\,\|g(t)\|_{\mathbb{H}^m_k} < \infty .
\end{equation*}
\end{theo}
 \noindent
Such regularity result is reminiscent from \cite[Theorem 5.4]{MoVi} but the method of proof is rather different and follows the robust strategy introduced recently in \cite{AloLo3} for the regularity of steady state for diffusively driven Boltzmann equation. Notice in particular that we shall invoke several regularity properties of the collision operator $\Q_e$ for generalized visco-elastic hard-spheres recently obtained in \cite{AloLo3}.  Second, uniform exponential pointwise lower bounds \eqref{villpoint} are needed for the solution $g(t,v)$.  The proof of such lower bound follows the steps developed in \cite{ada} and extended to granular gases for constant coefficient of normal restitution in \cite{MiMo2,MiMo3}.  The strategy was outlined in the pioneering work of Carleman \cite{Carle} and it is based on the \emph{spreading property} of the collision operator $\Q_e$ together with a repeated use of Duhamel's formula.  The non-autonomous nature of the collision operator $\Q_{\et}$ in the viscoleastic case will make this step the most technical of the paper.  It will not be sufficient to obtain a qualitative version of the spreading property of $\Q_{\et}$ but rather a \emph{quantitative estimate} giving explicit dependence on all the involved constants. This is a major difference with respect to the generalization of the result of \cite{ada} to the granular gases with constant coefficient of normal restitution performed in \cite{MiMo3}.  We dedicate the Section \ref{sec:low} to the proof of such fact.\\

\noindent
The organization of the paper is as follows: in Section \ref{sec:preli}, we briefly recall some known results on the free-cooling Boltzmann equation for hard-spheres interactions and establish some basics properties of the rescaled solution $g(t,v)$ that are needed in the sequel. In Section \ref{sec:condi} we provide a proof of Theorem \ref{theo:Main} assuming that the solution $g(t,v)$ to \eqref{eqgt} is sufficiently regular and uniformly bounded by below with an exponential function. In the next two sections we prove that the assumptions of such conditional result hold true, namely, in Section \ref{sec:reg} we prove Theorem \ref{theo:regu} while in Section \ref{sec:low} we prove Theorem \ref{lowVisco0}.  Section \ref{sec:discuss} discusses some possible improvements of the results and several technical details are recalled or established in the Appendix.
\section{Preliminaries: Elementary properties of the model}\label{sec:preli}
We state here for convenience several important results about the free cooling Boltzmann equation for hard spheres interactions.  They are taken from \cite{MMR,AloLo1,AloLo2}.  Let $f(\tau,w)$ be the unique solution to \eqref{be:force} associated to an initial datum $f_0$ satisfying \eqref{initial}. Assume that $e(\cdot)$ lies in the class $\mathcal{R}_\gamma$ for some $\gamma \geq 0$.  Recall that, from Theorem \ref{entrHaff}, if $f_0 \in L^p(\R^3)$ for some $p >1$, then the temperature
$$\E(\tau)=\IR f(\tau,w)\,|w|^2\d w$$
satisfies inequalities \eqref{Haff's} and \eqref{dEt}. Then, recalling the scaling \eqref{resca}, we easily deduce some of the properties of the scaling functions in \eqref{resca}-\eqref{V(t)}. Namely, one has the following
\begin{lemme} Let $s(\cdot)$ denote the inverse of the mapping $\tau \in \R^+ \mapsto t(\tau) \in \R^+$. For any $t \geq 0$, one has
\begin{equation} \label{s(t)}
\left(1+\frac{\gamma}{\gamma+1}\frac{t}{\sqrt{c_2}}\right)^{\frac{1+\gamma}{\gamma}}-1  \leq  s(t)  \leq \left(1+\frac{\gamma}{\gamma+1}\frac{t}{\sqrt{c_1}}\right)^{\frac{1+\gamma}{\gamma}}-1
\end{equation}
and
\begin{equation}\label{xi(t)}
{A_1} \sqrt{\E(0)}c_1^{\gamma/2}\left(1+\frac{\gamma}{1+\gamma}\frac{t}{\sqrt{c_1}}\right)^{-1}  \leq \xi(t) \leq \,A_2\, \sqrt{\E(0)} c_2^{\gamma/2}\left(1+\frac{\gamma}{1+\gamma}\frac{t}{\sqrt{c_2}}\right)^{-1}
\end{equation}
where $0 < c_1 \leq c_2 $ and $0 < A_1 \leq A_2$ are the constants appearing in \eqref{Haff's} and \eqref{dEt} respectively. Moreover,
$$\et(r)=e\left(z(t)r\right) \qquad \forall t > 0,\:\:r >0$$
with
\begin{equation}\label{sqrtEt}
\sqrt{c_1}\left(1+\frac{\gamma}{\gamma+1}\frac{t}{\sqrt{c_1}}\right)^{-\frac{1}{\gamma}} \leq z(t) \leq \sqrt{c_2}\left(1+\frac{\gamma}{\gamma+1}\frac{t}{\sqrt{c_2}}\right)^{-\frac{1}{\gamma}} \qquad \forall t > 0.\end{equation}
\end{lemme}
\begin{proof} A direct application of inequality \eqref{Haff's} shows that, for any $\tau \geq 0$, it holds
$$\sqrt{c_1} \frac{1+\gamma}{\gamma}\left((1+\tau)^{\frac{\gamma}{1+\gamma}} -1\right) \leq t(\tau) \leq \sqrt{c_2}\frac{1+\gamma}{\gamma}\left((1+\tau)^{\frac{\gamma}{1+\gamma}} -1\right).$$
Then, since $s(t)=\tau$ if and only if $t=t(\tau)$ one easily gets \eqref{s(t)}. In the same way, since
$\xi(t)=-\dfrac{ \sqrt{\E(0)}}{2}\dfrac{\dot{\E}(s(t))}{\E^{3/2}(s(t))}$, one deduces first from \eqref{dEt} that
$$\sqrt{\E(0)}\,A_1 \E^{\gamma/2}(s(t)) \leq \xi(t) \leq \sqrt{\E(0)}\,A_2 \E^{\gamma/2}(s(t))$$
which, using again \eqref{Haff's}, yields \eqref{xi(t)}. The estimate for $z(t)$ follows again simply from \eqref{Haff's}.
\end{proof}
\begin{nb} Notice that, according to \eqref{xi(t)}, for any $t, s \geq 0$, it holds
 \begin{equation*}\begin{split}
\int_t^{s+t} \xi(\tau)\d\tau &\leq A_2\,\sqrt{\E(0)}\,c_2^{\gamma/2}\int_t^{s+t}\left(1+\frac{\gamma}{1+\gamma}\frac{\tau}{\sqrt{c_2}}\right)^{-1}\d\tau\\
&\leq A_2\,\sqrt{\E(0)}\,c_2^{\gamma/2}\int_0^{s}\left(1+\frac{\gamma}{1+\gamma}\frac{\tau}{\sqrt{c_2}}\right)^{-1}\d\tau\\
&\leq A_2\,\sqrt{\E(0)}\,c_2^{\gamma/2} \dfrac{\sqrt{c_2}}{\sqrt{c_1}}\int_0^{s}\left(1+\frac{\gamma}{1+\gamma}\frac{\tau}{\sqrt{c_1}}\right)^{-1}\d\tau.
\end{split}\end{equation*}
We used the fact that $\tau \mapsto \left(1+\frac{\gamma}{1+\gamma}\frac{\tau}{\sqrt{c_2}}\right)^{-1}$ is non-increasing for the second inequality and that $c_2 \geq c_1$ for the latter.  Using again \eqref{xi(t)}, one gets that
\begin{equation}\label{intxit}\int_t^{s+t} \xi(\tau)\d\tau \leq \bar{c}\int_0^s \xi(\tau)\d \tau\,,\qquad \qquad \forall\; t,\,s \geq 0\end{equation}
for some universal constant $\bar{c}:=\dfrac{A_2}{A_1}\left(\dfrac{c_2}{c_1}\right)^{\frac{1+\gamma}{2}} \geq 1.$
\end{nb}
The propagation and creation of moments and the propagation of the $L^{p}$-norm is taken from \cite{MiMo2} in the constant case and from \cite{AloLo1} in the viscoelastic case
\begin{propo}\label{gtprop}
Let $f_0$ satisfying \eqref{initial} with $f_0\in  L^p(\R^3)$ for some $1 < p <\infty$.  Let $g(t,\cdot)$ be the solution to the rescaled equation \eqref{eqgt}. Then,
\begin{equation}\label{Thettau}
{\Theta}(t):=\IR g(t,v)|v|^2\d v=\IR f_0(w)|w|^2\d w =:\E_0
\end{equation}
and there exists a constant $L_p >0$ such that
\begin{equation}\label{gp}
\sup_{t\geq0}\left\|g(t)\right\|_{L^p} \leq L_p.
\end{equation}
More generally, for any $t_0 > 0$ and $k \geq 0$, there are $C_k > c_k > 0$ such that for all $t\geq t_0$
\begin{equation*}
c_k \leq \IR g(t,v)\,|v|^k \d v \leq C_k.
\end{equation*}
\end{propo}
\begin{proof} The fact that the temperature of $g(t,v)$ is constant follows from the scaling \eqref{resca}. Indeed, for any $\tau \geq 0$
$$\E(\tau)=\IR f(\tau,w)|w|^2\d w =V^3(\tau)\,\IR g(t(\tau),V(\tau)w)\,|w|^2\d w=V^{-2}(\tau)\,\IR g(t(\tau),v)\,|v|^2\d v$$
which, since $V^{-2}(\tau)=\E(\tau)/\E(0)$ shows \eqref{Thettau}. The rest of the proof follows from the analysis performed in \cite{AloLo1}.
\end{proof}

\section{Long time behavior: conditional proof of Theorem \ref{theo:Main}.}\label{sec:condi}

Let $g(t,v)$ be the unique solution to \eqref{eqgt}. One recalls that
$$\IR g(t,v)\d v=1,\qquad \IR v\,g(t,v)\d v=0 \qquad \forall t \geq 0$$
and
$$\IR |v|^2\,g(t,v)\d v = \E_0=\IR f_0(w)\,|w|^2\d w \qquad \forall t \geq 0.$$
In this section we prove the conditional version of Theorem \ref{theo:Main} in which we assume that everything goes well; i.e. the solution $g(t,\cdot)$ enjoys all the needed regularity and satisfies suitable uniform-in-time lower bounds:
\begin{hyp}\label{hyp:condi}
We assume that there exists $t_0 > 0$ and $\mathfrak{q} \geq 2$ such that
\begin{equation}\label{hyp:lower}
g(t,v) \geq a_0^{-1} \exp\left(-a_0\,|v|^\mathfrak{q}\right)\, \qquad \forall t \geq t_0\end{equation}
for some positive constant $a_0 > 0$.  Moreover, we assume that there exist $\ell  > 0$ and $k > 0$ large enough such that
\begin{equation}\label{hyp:reg}
\sup_{t \geq 0} \,\|g(t)\|_{\mathbb{H}^{\ell}_{k}}=:M(\ell,k) < \infty.\end{equation}
\end{hyp}
\noindent
Under such assumptions,  one has
\begin{theo}\label{main}
Assume that $e(\cdot)$ belong to the class $\mathcal{R}_{\gamma}$ with $\gamma>0$ and assume moreover that the rescaled solution $g(t,v)$ satisfies Assumptions \ref{hyp:condi}. Then, for any $\varepsilon \in (0,1)$ the following holds
\begin{equation*}
\left\|g(t)-\mathcal{M}_0\right\|_{L^1} \leq C_1 \xi(t)^{\frac{1}{2(1+\varepsilon)}} \leq C_2\,\left(1+\tfrac{\gamma}{1+\gamma}\,t\right)^{-\frac{1}{2(1+\varepsilon)}}\,, \qquad \forall t \geq 0
\end{equation*}
for some positive constants $C_1,C_2 > 0$ depending $\varepsilon$ (and the regularity of $g$) and where $\M_0$ is the Maxwellian distribution given by \eqref{maxwellianM0}.
\end{theo}
\noindent
Let us dedicate our efforts in proving Theorem \ref{main} in the remain of the section.
\subsection{Study of the relative entropy}
Let introduce the time dependent relative entropy
\begin{equation}\label{Ht}
\mathcal{H}(t)=\mathcal{H}\left(g(t)\,|\,\M_0\right):=\IR g(t,v)\log\left(\frac{g(t,v)}{\M_0(v)}\right)\d v, \qquad t \geq 0
\end{equation}
and define the entropy dissipation functional associated to the \emph{elastic} Boltzmann operator
\begin{equation}\label{D1}
\D_1(t)=\D_1(g(t)\,|\,\M_0):=-\IR \Q_{1} \big(g,g\big)(t,v)\log\left(\frac{g(t,v)}{\M_0(v)}\right)\d v.
\end{equation}
One has the following
\begin{lemme}\label{eq:DH}
The evolution of $\mathcal{H}(t)$ is given by the following
\begin{equation}\label{dHt}
\dfrac{\d}{\d t}\mathcal{H}(t)+\D_1(t)=\mathcal{I}_1(t)+\mathcal{I}_2(t) \,,\qquad \forall t \geq 0
\end{equation}
with
\begin{equation}\begin{split}\label{I1}
\mathcal{I}_{1}(t):=&- {\xi(t)} \IR \nabla_{v}\cdot\big( v\,g(t,v) \big)\log\left(\frac{g(t,v)}{\M_0(v)}\right)\d v\,,\\
\mathcal{I}_{2}(t):=&\IR \left( \Q_{\et}^+\big(g,g\big)-\Q_1^+\big(g,g\big) \right)(t,v) \log\left(\frac{g(t,v)}{\M_0(v)}\right)\d v.
\end{split}
\end{equation}
where $\xi(t) \propto \left(1+\frac{\gamma}{1+\gamma}t\right)^{-1}$ has been defined in \eqref{xitau}.
\end{lemme}
\begin{proof}
Let $h(t,v)=g(t,v)\log\left(\frac{g(t,v)}{\M_0(v)}\right)=g(t,v)\log g(t,v)-g(t,v)\log \M_0(v).$ One deduces from \eqref{eqgt}  that
\begin{equation*}\begin{split}
\partial_t h(t,v)&=\partial_{t}g(t,v)
 +\partial_t g(t,v)\left(\log\left(\frac{g(t,v)}{\M_0(v)}\right)\right)
\\
&=\partial_{t}g(t,v)
- {\xi(t)} \nabla_{v}\cdot\big(v\,g(t,v)\big)\log\left(\frac{g(t,v)}{\M_0(v)}\right)\\
&\phantom{+++++++} +\left(\Q_{\et}\left(g,g\right)-\Q_1\left(g,g\right)\right)(t,v) \log\left(\frac{g(t,v)}{\M_0(v)}\right)\\
&\phantom{+++++++++++++++++} +\Q_1\left(g,g\right) (t,v) \log\left(\frac{g(t,v)}{\M_0(v)}\right).
\end{split}
\end{equation*}
Integrating this identity over $\mathbb{R}^3$ yields the result after using that
$\Q_{\et}^-(g,g)=\Q_1^-(g,g).$
\end{proof}
\noindent
Let us proceed by estimating the terms $\mathcal{I}_1(t)$ and $\mathcal{I}_2(t)$.
\begin{lemme}\label{lem:I1}
Under Assumption \ref{hyp:condi}, there exists $C > 0$ such that
\begin{equation}\label{estimI1}
\left|\mathcal{I}_1(t)\right| \leq C\xi(t)\|g(t)\|_{\mathbb{W}^{1,1}_{\mathfrak{q}+1}}\,\log\|g(t)\|_\infty\, \qquad \forall t \geq t_0,
\end{equation}
and
\begin{equation}\label{estimI2}
\left|\mathcal{I}_2(t)\right| \leq C \xi(t) \|g(t)\|_{L^1_{\mathfrak{q}+\gamma+4}}\|g(t)\|_{\mathbb{H}^{1}_{\mathfrak{q}+\gamma+4}}\,\log\|g(t)\|_\infty\, \qquad \forall t \geq t_0.
\end{equation}
\end{lemme}
\begin{proof} In the sequel $C$ is a positive constant possibly changing from line to line.  First, we prove the estimate for $\mathcal{I}_1(t)$.  Fix $t_0 > 0$ and observe that according to \eqref{hyp:lower} there exists $C > 0$ such that
\begin{equation*}
\left|\log g(t,v)\right| \leq  \log\|g(t)\|_\infty\,+ C\left(1+|v|^2\right)^{\frac{\mathfrak{q}}{2}} \qquad \forall t \geq t_0
\end{equation*}
from which we deduce that
\begin{equation}\label{g:M}
\left|\log\left(\frac{g(t,v)}{\M_0(v)}\right)\right | \leq  C\log\|g(t)\|_\infty\, \langle v \rangle^{\mathfrak{q}} \qquad \forall t \geq t_0
\end{equation}
and
\begin{equation*}
\left|\mathcal{I}_1(t)\right| \leq  C\,\xi(t)\log\|g(t)\|_\infty\IR \left|\nabla_{v}\cdot(v\,g(t,v))\right|\langle v \rangle^{\mathfrak{q}}\d v
\end{equation*}
which proves \eqref{estimI1}. Second, regarding estimate \eqref{estimI2} one notices that, thanks to \eqref{g:M}
\begin{equation*}
\left|\mathcal{I}_2(t)\right| \leq  C\log\|g(t)\|_\infty\, \left\|\Q_{\et}^+(g,g)(t)-\Q_1^+(g,g)(t)\right\|_{L^1_{\mathfrak{q}}} \qquad \forall t \geq t_0.
\end{equation*}
Applying Proposition \ref{diffL1k} with $\lambda=\xi(t)^{1/\gamma}$ and $\theta=1$ one obtains that
\begin{equation*}
\left\|\Q_{\et}^+(g,g)(t)-\Q_1^+(g,g)(t)\right\|_{L^1_{\mathfrak{q}}} \leq C \xi(t)\|g(t)\|_{L^1_{\mathfrak{q}+\gamma+4}}\|g(t)\|_{\mathbb{H}^{1}_{\mathfrak{q}+\gamma+4}}
\end{equation*}
which proves the estimate.
\end{proof}
\noindent
The use of Villani's Theorem \ref{theo:Vill}, licit under Assumptions \ref{hyp:condi}, yields.
\begin{propo}\label{converEntropy}
For any $t_0 > 0$ and any $\varepsilon  \in (0,1)$,
\begin{equation}\label{decayentropy}
\mathcal{H}(t) \leq C\,\xi(t)^{\frac{1}{1+\varepsilon}}\,, \qquad \forall t \geq t_0
\end{equation}
for some positive constant $C$ depending on $\varepsilon$, $M=M(\ell_\varepsilon,k_\varepsilon)$ ($\ell_\varepsilon$ and $k_\varepsilon$ are large enough satisfying \eqref{hyp:reg}), $r$ and $\mathcal{H}(t_0)$.
\end{propo}
\begin{proof} Thanks to Assumptions \ref{hyp:condi}, one can apply Theorem \ref{theo:Vill} to assert that, for any $\varepsilon >0$ and $t \geq t_0$ fixed, there exist $\ell_\varepsilon > 0$ and $k_\varepsilon > 1$ such that if
\begin{equation*}
\sup_{t \geq 0} \,\|g(t)\|_{\mathbb{H}^{\ell_\varepsilon}_{k_\varepsilon}} \leq  M,
\end{equation*}
then
\begin{equation*}
\D_1(t) \geq C_\varepsilon\,\mathcal{H}(t)^{1+\varepsilon}\,, \qquad \forall t \geq t_0
\end{equation*}
with $C_\varepsilon$ depending only on $\varepsilon$ and $M$. Hence, we can combine \eqref{dHt} with Lemma \ref{lem:I1} to obtain
\begin{multline*}
\dfrac{\d}{\d t}\mathcal{H}(t) + C_\varepsilon \,\mathcal{H}(t)^{1+\varepsilon} \leq  C_1 \xi(t) \|g(t)\|_{L^1_{\mathfrak{q}+\gamma+4}}\|g(t)\|_{\mathbb{H}^{1}_{\mathfrak{q}+\gamma+4}}\,\log\|g(t)\|_\infty \\
+ C_2\xi(t)\|g(t)\|_{\mathbb{W}^{1,1}_{\mathfrak{q}+1}}\,\log\|g(t)\|_\infty,\quad \forall t\geq t_0.
\end{multline*}
Therefore,
\begin{equation}\label{dHtVare}
\dfrac{\d}{\d t}\mathcal{H}(t) + C_\varepsilon(g)\,\mathcal{H}(t)^{1+\varepsilon} \leq C \, \xi(t)
\end{equation}
with
\begin{equation*}
C=\max(C_1,C_2)\sup_{t \geq t_0}\left(\|g(t)\|_{\mathbb{H}^{1}_{\mathfrak{q}+\gamma+4}} +\|g(t)\|_{L^1_{\mathfrak{q}+\gamma+4}}+\|g(t)\|_{\mathbb{W}^{1,1}_{\mathfrak{q}+1} }+\log\|g(t)\|_\infty\right).
\end{equation*}
After possibly increasing $\ell_\varepsilon$ and $k_\varepsilon$, one sees that $M < \infty$ implies $C < \infty$.  Hence,
\begin{equation*}
\dfrac{\d}{\d t}\mathcal{H}(t) + C_\varepsilon \,\mathcal{H}(t)^{1+\varepsilon} \leq C \, \xi(t) \,, \qquad \forall t \geq t_0
\end{equation*}
and therefore, $\mathcal{H}(t) \leq C(\varepsilon,\mathcal{H}(t_0))  \, \xi(t)^{\frac{1}{1+\varepsilon}}\,,$  $\forall t \geq t_0.$\end{proof}

\noindent
We are now in position to prove Theorem \ref{main}.
\begin{proof}[Proof of Theorem \ref{main}]  From Proposition \ref{converEntropy} and Csisz\'{a}r-Kullback-Pinsker inequality we get that
\begin{equation*}
\|g(t)-\M_0\|_{L^1} \leq C_0\,\xi(t)^{\frac{1}{2(1+\varepsilon)}}\, \qquad \forall t \geq t_0,
\end{equation*}
with $C_0$ depending on $\epsilon$ and $M$ yielding the result.\end{proof}
\section{Regularity of the rescaled solution}\label{sec:reg}
The objective of this section is to prove propagation of regularity for the rescaled solution to equation \eqref{eqgt}, therefore, showing that assumption \eqref{hyp:reg} holds.  We will need to add the following smoothness condition on the coefficient of normal restitution.
\begin{hyp}\label{HYPdiff}
We assume that the coefficient of normal restitution belongs to the class $\mathcal{R}_\gamma$ with $\gamma > 0$.  Additionally, it is infinitely differentiable over $(0,\infty)$  with
\begin{equation*}\label{assm}
\sup_{r \geq 0}\, r \, e^{(k)}(r) < \infty \qquad \forall k \geq 1.
\end{equation*}
The symbol $e^{(k)}(\cdot)$ denotes the $k^{th}$ order derivative of $e(\cdot)$.
\end{hyp}
\noindent
We begin with the following simple generalization of Proposition \ref{gtprop}, taken from \cite[Lemma 5.8 and Proposition 5.9]{AloLo1}:
\begin{propo}\label{propLpeta}
Assume that the initial datum $f_0 \geq 0$ satisfies the  conditions given by \eqref{initial} with $f_0 \in L^{p}(\R^3)$ for some $1<p<\infty$ and let $g(\tau,\cdot)$ be the solution to the rescaled equation \eqref{eqgt} with initial datum $g(0,\cdot)=f_0$. Then, there exists a constant $\kappa_0 >0$ such that
\begin{equation*}
\IR g(t,v)|v-\vb|\d\vb\geq \kappa_0 \langle v \rangle, \qquad \forall v \in \R^3, \quad t >0.
\end{equation*}
In particular,
\begin{equation*}
\int_{\R^3}
g^{p-1}\Q^-_e(g,g)\d v \geq \kappa_0\IR g^p(t,v)(1+|v|^2)^{1/2}\d v=\kappa_0  \left\|g(t)\right\|_{L^p_{1/p}}^p.
\end{equation*}
Moreover, if $f_0 \in L^{1}_{2(1+\eta)}\cap L^p_{\eta}(\R^3)$ for some $p > 1$ and $\eta \geq 0$, then
\begin{equation*}
\sup_{t \geq 0} \left\|g(t)\right\|_{L^p_\eta}  < \infty.
\end{equation*}
\end{propo}
\begin{nb} We shall use Proposition \eqref{propLpeta} for the case $p=2$.  Notice that, as it is the case for elastic Boltzmann equation \cite{MoVi}, it is possible also to prove the \emph{appearance} of $L^2$-moments;  if $f_0 \in L^1_2 \cap L^2$ then, for any $t_0 > 0$ and $k \geq 0$
\begin{equation*}
\sup_{t \geq t_0}\|g(t)\|_{L^2_k} < \infty.
\end{equation*}
\end{nb}
\noindent
The goal now is to extend Proposition \eqref{propLpeta} to weighted Sobolev norms.  We adopt the strategy applied to the regularity of steady state for diffusively driven Boltzmann equation in \cite{AloLo3}, see Theorem \ref{regularite} in the Appendix.

\begin{propo}\label{propo:reg}
Let $s \in \mathbb{N}$ and $k \geq 1/2$ be fixed.  If $f_0 \in \mathbb{H}^{s+1}_k \cap L^1_2$ and
\begin{equation}\label{Hsk}
\sup_{t\geq 0}\left(\|g(t)\|_{\mathbb{H}^s_{2k+s+2}} + \| g(t)\|_{L^1_{2k+s+2}}\right)< \infty,
\end{equation}
then
\begin{equation*}
\sup_{t \geq 0}\|g(t)\|_{\mathbb{H}_k^{s+1}} < \infty.
\end{equation*}
\end{propo}
\begin{proof} For $\ell \in \mathbb{N}^3$ with $|\ell|=s+1$, we set $G_\ell(t,v)=\partial^\ell_v g(t,v)$ which satisfies
\begin{equation*}
\partial_t G_\ell(t,v) + 3\xi(t)G_\ell(t,v) + \xi(t) \partial^\ell_v \big(v \cdot \nabla_v g(t,v)\big)=\partial^\ell \Q_{\et}\big(g,g\big)(t,v).
\end{equation*}
Using the identity
\begin{equation*}
\partial^\ell_v \big(v \cdot \nabla_v  g(t,v)\big)=v \cdot \nabla_v  G_\ell(t,v)+ |\ell| G_\ell(t,v),
\end{equation*}
we obtain
\begin{equation*}
\partial_t G_\ell(t,v) +\left(3\xi(t)+|\ell|\,\right)G_\ell(t,v)+  \xi(t) v \cdot \nabla_v G_\ell(t,v)=\partial^\ell \Q_{\et}\big(g,g\big).
\end{equation*}
Fix $k \geq 1/2$, multiply this last equation by $G_\ell(t,v) \langle v\rangle^{2k}$ and integrate over $\R^3$ to obtain
 \begin{equation*}
\begin{split}
\IR \langle v \rangle^{2k}\big( v \cdot \nabla_v G_\ell(t,v)\big)G_\ell(t,v)\d v&=-\frac{1}{2}\IR \mathrm{div}\big(v\langle v \rangle^{2k}\big)G_\ell^2(t,v)\d v\\
&=-\frac{3+2k}{2}\|G_\ell(t)\|_{L^2_k}^2 + k \|G_\ell(t)\|_{L^2_{k-1}}^2.
\end{split}
\end{equation*}
Therefore,
\begin{multline}\label{eq:Gl}
\dfrac{1}{2}\dfrac{\d }{\d t}\|G_\ell(t)\|_{L^2_k}^2 + \left( {3\xi(t)} +|\ell|\,\right)\|G_\ell(t)\|_{L^2_k}^2 +  {k\xi(t)} \, \|G_\ell(t)\|_{L^2_{k-1}}^2
\\
=\IR \partial^\ell \Q_{\et}\big(g,g\big)(t,v)G_\ell(t,v)\langle v \rangle^{2k}\d v +  {k\xi(t)} \|G_\ell(t)\|_{L^2_k}^2.
 \end{multline}
The first integral in the right side can be estimated as
\begin{equation*}
\begin{split}
\IR \partial^\ell \Q^+_{\et}\big(g,g\big)(t,v)\gl(t,v)\langle v\rangle^{2k}\d v &\leq \|\partial^\ell \Q^+_{\et}\big(g(t,\cdot),g(t,\cdot)\big)\|_{L^2_{k-\frac{1}{2}}}\|\gl(t)\|_{L^2_{k+\frac{1}{2}}}\\
&\leq \|\Q^+_{\et}\big(g(t,\cdot),g(t,\cdot)\big)\|_{\mathbb{H}^{s+1}_{k-\frac{1}{2}}}\|\gl(t)\|_{L^2_{k+\frac{1}{2}}}
\end{split}
\end{equation*}
since $|\ell|=s+1$.  The Sobolev norm of $\Q^+_{\et}(g,g)$ can be estimated thanks to Theorem \ref{regularite} applied to $\eta=k-\frac{1}{2}$.  More precisely, for any $\varepsilon >0$
\begin{multline}\label{Qet+gt}
\|\Q^+_{\et}\big(g(t,\cdot),g(t,\cdot)\big)\|_{\mathbb{H}^{s+1}_{k-\frac{1}{2}}} \leq C(\varepsilon)\,\|g(t)\|_{\mathbb{H}^s_{2k+s+2}}\|g(t)\|_{L^1_{2k+s+2}} + \varepsilon\| g(t)\|_{\mathbb{H}^s_{k+\frac{5}{2}}}\,\|g(t)\|_{\mathbb{H}^s_{k+\frac{1}{2}}}\\
+2\varepsilon \|g(t)\|_{L^1_{k+\frac{1}{2}}}\,\|\gl(t)\|_{L^2_{k+\frac{1}{2}}}.
\end{multline}
Using \eqref{Hsk}, one observes that  there exist positive $\alpha_k=\alpha_k(\varepsilon)$ and $\beta_k$ such that
\begin{equation*}
\|\Q^+_{\et}\big(g(t,\cdot),g(t,\cdot)\big)\|_{\mathbb{H}^{s+1}_{k-\frac{1}{2}}} \leq \alpha_k +  \varepsilon\,\beta_k\,  \|G_\ell(t)\|_{L^{2}_{k+\frac{1}{2}}}\,, \qquad \qquad \forall t \geq 0.
\end{equation*}
Therefore,
\begin{equation}\label{Q+el}
\IR \partial^\ell \Q^+_{\et}\big(g,g\big)(t,v)\gl(t,v)\langle v\rangle^{2k}\d v \leq  \alpha_k\|\gl(t)\|_{L^{2}_{k+\frac{1}{2}}} + \varepsilon\,\beta_k\, \|\gl(t)\|_{L^{2}_{k+\frac{1}{2}}}^2.
\end{equation}
Regarding the loss part of the collision operator note that
\begin{equation*}
\partial^{\ell}\Q^{-}_{\et}\big(g,g\big)=\sum^{\ell}_{\nu=0}
\left(\begin{array}{c}
\ell \\ \nu
\end{array}\right)
\Q_1^{-}\big(\partial^\nu g,\partial^{\l-\nu}g\big).
\end{equation*}
For any $\nu$ with $\nu\neq \ell$, there exists $i_0 \in \{1,2,3\}$ such that $\ell_{i_0}-\nu_{i_0} \geq 1$ and integration by parts yields
\begin{equation*}\begin{split}
\left|\Q^{-}_1\big(\partial^\nu g,\partial^{\l-\nu}g\big)(t,v)\right|&=\left|\partial^\nu g(t,v)\right|\,\left|\IR \partial^{\l-\nu} g(t,\vb)|v-\vb|\d\vb\right|\\
&\leq \left|\partial^{\,\nu} g(t,v)\right|\,\|\partial^{\,\sigma }g(t)\|_{L^1}
\end{split}\end{equation*}
where $\sigma=(\sigma_1,\sigma_2,\sigma_3)$ is defined with $\sigma_{i_0}=\ell_{i_0}-\nu_{i_0}-1$ and $\sigma_i=\ell_i-\nu_i$ if $i \neq i_0.$  Thus, estimating the weighted $L^1$-norm by an appropriate weighted $L^2$-norm (see \eqref{taug21}) we obtain
\begin{equation*}
\left|\Q^{-}_{1}\big(\partial^\nu g,\partial^{\l-\nu}g\big)(t,v)\right|\leq C\,|\partial^\nu g(t,v)|\,\|\partial^{\,\sigma }g(t)\|_{L^2_2}
\end{equation*}
for some universal constant $C >0$ independent of $t$.  According to \eqref{Hsk} this last quantity is uniformly bounded.  Hence, using Cauchy-Schwarz inequality we conclude that
\begin{multline}\label{Q-nu}
\underset{\nu \neq \ell}{\sum_{\nu =0}^{\ell}}
\left(\begin{array}{c}
\ell \\ \nu
\end{array}\right)\IR  \Q^{-}_{1}\big(\partial^\nu g,\partial^{\l-\nu}g\big)(t,v) \,\gl(t,v)\langle v\rangle ^{2k}\d v
\\
\leq C_2\sum_{|\nu|< |\ell|}
\left(\begin{array}{c}
\l \\ \nu
\end{array}\right)\|\partial^\nu g(t)\|_{L^2_k}\,\|\gl(t)\|_{L^2_k}
\leq C_{k,\l}\|\gl(t)\|_{L^2_k} \qquad\forall t \geq 0
\end{multline}
for some positive constant $C_{k,\l}$ independent of $t$.  Whenever $\nu=\ell$ one has
\begin{equation*}
\int_{\mathbb{R}^{3}}\Q^{-}_{1}\big(\partial^{\ell} g,g\big)(t,v)\,\partial^{\l}g(t,v)\langle v\rangle ^{2k}\d v=\IR \gl^2(t,v)\langle v\rangle^{2k}\d v\IR g(t,\vb)\,|v-\vb|\d \vb,
\end{equation*}
thus, thanks to Proposition \ref{propLpeta} one obtains the lower bound
\begin{equation}\label{Q-nu+}
\int_{\mathbb{R}^{3}}\Q^{-}_{1}\big(\partial^{\ell} g,g\big)(t,v)\,\partial^{\l}g(t,v)\langle v\rangle ^{2k}\d v\geq \kappa_0\,\|\gl(t)\|^{2}_{L^2_{k+\frac{1}{2}}}.
\end{equation}
Combining equation \eqref{eq:Gl} with \eqref{Q+el}, \eqref{Q-nu} and \eqref{Q-nu+}, one obtains
\begin{multline}\label{eq:Gl1}
\dfrac{\d }{\d t}\|G_\ell(t)\|_{L^2_k}^2 + \left(3\xi(t)+2|\ell|\,\right)\|G_\ell(t)\|_{L^2_k}^2 +  {2k} \xi(t)\, \|G_\ell(t)\|_{L^2_{k-1}}^2
\\
\leq 2\alpha_k \|\gl(t)\|_{L^2_{k+1/2}} + 2\varepsilon\,\beta_k\,\|\gl(t)\|_{L^2_{k+1/2}}^2 \\
 +C_{k,\ell}\,\|\gl(t)\|_{L^2_k} -2\kappa_0\|\gl(t)\|_{L^2_{k+1/2}}^2 +   {2k} \xi(t)\|G_\ell(t)\|_{L^2_k}^2.
 \end{multline}
Choosing $\varepsilon >0$ such that $\varepsilon\,\beta_k=\frac{\kappa_0}{2}$ and controlling all $L^2_k$-norms by $L^2_{k+\frac{1}{2}}$-norms, there exists some positive constant $A_k >0$ such that
\begin{equation*}\label{ukk}
\dfrac{\d }{\d t}u_k^2(t)  +    \kappa_0 u_{k+1/2}^2(t)  \leq A_k\,u_{k+1/2}(t) +  (2k-3)C\, \,u_k^2(t) \qquad \forall t \geq 0\end{equation*}
where we used that $\xi(t) \leq C=A_2\, \sqrt{\E(0)} c_2^{\gamma/2}$ and set for simplicity
\begin{equation*}
u_k(t):=\|\gl(t)\|_{L^2_k}.
\end{equation*}
Using Young's inequality $u_{k+1/2}(t) \leq\epsilon u_{k+1/2}^2(t) + \frac{1}{4\epsilon}$ for any $\epsilon >0$.  Thus, by choosing $\epsilon A_k=\frac{\kappa_0}{2}$, there exist $C_k > 0$ such that
\begin{equation}\label{ukk}
\dfrac{\d }{\d t}u_k^2(t)  +   \frac{\kappa_0}{2}u_{k+1/2}^2(t)  \leq C_k + {(2k-3)}\,C\,  \,u_k^2(t) \qquad \forall t \geq 0.
\end{equation}
One distinguishes two cases according to the sign of $2k-3$:\\

\noindent
\textit{When $k \leq 3/2$}, using the fact that $u_{k+1/2}(t) \geq u_k(t)$, inequality \eqref{ukk} yields
\begin{equation*}
\dfrac{\d }{\d t}u_k^2(t)  + \frac{\kappa_0}{2}u_{k}^2(t)  \leq C_k
\end{equation*}
and clearly,
\begin{equation*}
\sup_{t \geq 0} u_k(t) \leq \max\left\{u_k(0),\sqrt{\frac{2C_k}{\kappa_0}}\right\}.
\end{equation*}
\\

\noindent
\textit{When $k > 3/2$}, one cannot neglect anymore the last term in \eqref{ukk}.  For any $R >0$ simple interpolation leads to the estimate
\begin{equation*}
\|\gl(t)\|_{L^2_k}^2 \leq R^{-1}\|\gl(t)\|_{L^2_{k+1/2}}^2 + R^{2k-3}\|\gl(t)\|_{L^2_{3/2}}^2.
\end{equation*}
From the previous step, the $L^2_{3/2}$-norm of $\gl(t)$ is uniformly bounded. Thus, choosing $R$ large enough so that $C(2k-3)R^{-1}=\frac{\kappa_0}{4}$, one obtains from \eqref{ukk} that
\begin{equation*}
\dfrac{\d }{\d t}u_k^2(t)  +   \dfrac{\kappa_0}{4}u_{k+1/2}^2(t)  \leq C_k
\end{equation*}
for some positive constant $C_k >0$. Arguing as before,
\begin{equation*}
\sup_{t \geq 0}\, u_k(t) \leq \max\left\{u_k(0),\sqrt{\frac{4C_k}{\kappa_0}}\right\}.
\end{equation*}
This completes the proof.
\end{proof}
\noindent
We are now in position to prove our main regularity result Theorem \ref{theo:regu}

\begin{proof}[Proof of Theorem \ref{theo:regu}] The proof is a simple consequence of Propositions \ref{propLpeta} and \ref{propo:reg}. \\

\textit{First step ($m=0$).}  For a given $k \geq 0$ Proposition \ref{propLpeta} asserts that $\sup_{t \geq 0}\|g(t)\|_{L^2_k} < \infty$  provided $f_0 \in L^2_k \cap L^1_{2k+2}$, that is, $\eta_0(k,0)=0$ and $\eta_1(k,0)=k+2.$\\

\textit{Second step ($m\geq1$).}  Set $m=1$ and fix $k \geq 0$, according to Proposition \ref{propo:reg} one has $\sup_{t \geq 0}\,\|g(t)\|_{\mathbb{H}^1_k} < \infty$ provided $f_0 \in \mathbf{H}^1_k$ and $\sup_{t \geq 0}\,\left(\|g(t)\|_{L^2_{2k+2}} +\|g(t)\|_{L^1_{2k+2}}\right) < \infty$.  According to the first step, this holds true if $f_0 \in L^2_{2k+2} \cap L^1_{4k+6}$, that is, $\eta_0(k,1)=k+2$ and $\eta_1(k,1)=3k+6$.  For $m \geq 2$, the result follows in a similar way, constructing $\eta_0(k,m)$ and $\eta_1(k,m)$ by induction.
\end{proof}
\section{Pointwise lower bound: Proof of Theorem \ref{lowVisco0}}\label{sec:low}
The goal of this section is to prove precise lower bounds estimates of the rescaled solution to equation \eqref{eqgt} and, therefore, showing that assumption \eqref{hyp:lower} holds. As mentioned in the Introduction, the proof of such lower bounds follows the steps already developed in \cite{ada} and extended to granular gases in \cite{MiMo2,MiMo3}.
\subsection{Spreading properties of $\Q_e^+$}
The following proposition gives a notion of the diffusive effect of the gain Boltzmann collision operator for inelastic interactions.
\begin{propo}\label{spread}
For any $v_0 \in \mathbb{R}^3$ and any $\delta >0$, one has
\begin{equation*}
\mathrm{Supp}\left(\Q_e^+\big(\mathbf{1}_{\mathds{B}(v_0,\delta)}\,;\,\mathbf{1}_{\mathds{B}(v_0,\delta)}\big)\right)={\mathds{B}(v_0,\ell_e(\delta))}
\end{equation*}
where
\begin{equation*}
\ell_e(\delta)=\delta\,\sqrt{1+\beta_e^2(\delta)}  \in \left(\tfrac{\sqrt{5}}{2}\delta,\sqrt{2}\delta\right).
\end{equation*}
More precisely, for any $0 < \chi < 1$, there exists a universal $\kappa>0$ such that
\begin{equation}\label{bound}
\Q_e^+\big(\mathbf{1}_{\mathds{B}(v_0,\delta)}\,;\,\mathbf{1}_{\mathds{B}(v_0,\delta)}\big) \geq \kappa \,\delta^4 \chi^9 K_e^9(\delta)  \mathbf{1}_{\mathds{B}(v_0,(1-\chi)\ell_{e}(\delta))}  \qquad \forall \delta >0,
\end{equation}
where
\begin{equation*}
K_e(\delta):=\frac{ \delta}{ \vartheta_e^{-1}(\delta)\,\mathrm{Lip}_{[0,\delta]}(\vartheta^{-1}_{e})}, \qquad \delta >0.
\end{equation*}
Here, for a given Lipschitz function $f:I \to \mathbb{R}$, $\mathrm{Lip}_I(f)$ denotes the Lipschitz constant of $f$ over the interval $I$.
\end{propo}
\begin{nb}\label{nbImpo}
For general coefficient of normal restitution $e(\cdot)$ decreasing and such that $\vartheta_e(\cdot)$ is non decreasing the mapping $K_e\::\: \delta \geq 0 \longmapsto K_e(\delta) \in [0,1]$  is non increasing.  Indeed, writing
\begin{equation*}
K_e(\delta)=\dfrac{e\left(\vartheta_e^{-1}(\delta)\right)}{\mathrm{Lip}_{[0,\delta]}(\vartheta_e^{-1})}
\end{equation*}
one notices that since  $\vartheta_e^{-1}$ is non decreasing the mapping  $\delta \mapsto \mathrm{Lip}_{[0,\delta]}(\vartheta_e^{-1})$ is non decreasing.  And, since $e(\cdot)$ is decreasing the mapping $\delta \mapsto e(\vartheta_e^{-1}(\delta))$ is decreasing as well.
\end{nb}
\begin{nb}\label{rem:Cases}
Let us comment a bit the above lower bound exploring the various cases we have in mind.
\begin{enumerate}
\item A first remark is that the exponent $9$ in \eqref{bound}  is not optimal and can be improved with some effort up to $7$.\\
\item In the case in which $e(r) \geq e_0>0$ for any $r \geq 0$, one checks that there exists a positive constant $c_0 $ such that
\begin{equation*}
\vartheta_e^{-1}(\delta) \leq c_0 \delta \qquad \text{ and } \quad \mathrm{Lip}_{[0,\delta]}(\vartheta_e^{-1}) \leq c_0 \qquad \forall \delta > 0.
\end{equation*}
Therefore, $K_e(\cdot)$ is bounded away from zero and \eqref{bound} reads
\begin{equation*}
\Q_e^+\big(\mathbf{1}_{\mathds{B}(v_0,\delta)}\,;\,\mathbf{1}_{\mathds{B}(v_0,\delta)}\big) (v)\geq \kappa_0 \,\delta^4 \chi^9\, \mathbf{1}_{\mathds{B}(v_0,(1-\chi)\ell_{e}(\delta))}(v) \qquad \forall \delta >0,
\end{equation*}
for some universal constant $\kappa_0 > 0$ (independent of $\delta$ and $\chi$). This is the case for instance whenever $e(r)$ is constant .  In this instance \eqref{bound} can be seen as a quantitative improvement of the result of \cite{MiMo2}.\\
\item For viscoelastic hard spheres, the coefficient of normal restitution satisfies
\begin{equation}\label{visco}
e(r) + \mathfrak{a}\, r^{1/5}e(r)^{3/5}=1 \qquad \forall\; r\geq0
\end{equation}
for some given $\mathfrak{a} > 0.$  Since $\lim_{r \to \infty}e(r)=0$, one checks easily that $\vartheta_e^{-1}(\varrho) \simeq \mathfrak{a}^{\tfrac{5}{2}}\,\varrho^{\tfrac{3}{2}}$ for $\varrho \to \infty$ and therefore, there exists some positive constant $c_0$ (depending only on $\mathfrak{a}$) such that
\begin{equation*}
\vartheta_e^{-1}(\varrho) \leq c_0\,(1+\varrho)^{3/2} \qquad \forall \varrho \geq 0.
\end{equation*}
Using \eqref{visco} and estimating the derivative of $\vartheta_e^{-1}$, one can prove that there exists some $c_1 >0$ (depending only on $\mathfrak{a}$) such that
\begin{equation*}
\mathrm{Lip}_{[0,\delta]}(\vartheta_e^{-1}) \leq c_1\,\big(1+\delta\big)^{1/2} \qquad \forall \delta > 0.
\end{equation*}
In such case, $\delta \geq 1\mapsto \delta K_e(\delta)$ is bounded away from zero and \eqref{bound} translates into
\begin{equation*}
\Q_e^+\big(\mathbf{1}_{\mathds{B}(v_0,\delta)}\,;\,\mathbf{1}_{\mathds{B}(v_0,\delta)}\big) (v)\geq \kappa_1(\mathfrak{a})\chi^9 \,\delta^{-5} \mathbf{1}_{\mathds{B}(v_0,(1-\chi)\ell_{e}(\delta))}(v) \quad \forall\; \delta \geq 1,
\end{equation*}
for some constant $\kappa_1(\mathfrak{a})$ independent of $\delta$.
\item More generally, if there exists $m \in \mathbb{Z}$ and $C_m > 0$ such that \eqref{assKe} holds, that is,
\begin{equation*}
\vartheta_e^{-1}(\varrho) +\dfrac{\d}{\d \varrho}\vartheta_e^{-1}(\varrho) \leq C_m \big(1+\varrho\big)^m \qquad \forall \varrho > 0
\end{equation*}
then, one sees that the mapping $\delta \mapsto \mathrm{Lip}_{[0,\delta]}(\vartheta_e^{-1})$ is growing at most polynomially and in particular, there exists $n \in \mathbb{N}$ and $C > 0$ such that $\vartheta_e^{-1}(\delta)\mathrm{Lip}_{[0,\delta]}(\vartheta_e^{-1}) \leq C\delta^n$ for any $\delta \geq1$.  As a consequence,
\begin{equation}\label{Kee}K_e(\delta) \geq \frac{1}{C} \delta^{1-n} \qquad \forall\; \delta \geq 1.
\end{equation}
\end{enumerate}
\end{nb}
\noindent
The fact that the support of $\Q_{e}^{+}\big(\mathbf{1}_{\mathds{B}(v_0,\delta)}\,,\,\mathbf{1}_{\mathds{B}(v_0,\delta)}\big)$ is given by ${\mathds{B}(v_0,\ell_e(\delta))}$ is simple to show using the argument given in \cite{MiMo2}.  However, proving the lower bound \eqref{bound} is quite involved and will require several steps. We begin by noticing that, using invariance by translation, it is enough to prove the result for $v_0=0$.\\

\noindent
Let us first consider the simpler case in which $\delta=1$.  For the general case we use the scaling properties of $\Q_{e}^{+}$ to conclude.  Let us introduce
\begin{equation*}
F_e(v)=\Q_{e}^{+}\big(\mathbf{1}_{\mathds{B}(0,1)}\,;\,\mathbf{1}_{\mathds{B}(0,1)}\big)(v).
\end{equation*}
\begin{lemme}[Small velocities]\label{Fesmall} There exists some universal constant $c_0 >0$ such that
\begin{equation}\label{eqlowvel}
F_{e}(v)\geq c_0 {\mathbf{1}}_{\mathds{B}\left(0,\sqrt{\frac{3}{8}}\right)}(v).
\end{equation}

\end{lemme}
\begin{proof}
Using the strong formulation of the collision operator
\begin{equation}\label{strong}
F_e(v)=\It \dfrac{1}{e(|'\!u\cdot n|)J_e(|'\!u\cdot n|)}\mathbf{1}_{\mathds{B}(0,1)}('\!v)\mathbf{1}_{\mathds{B}(0,1)}('\!\vb)|u\cdot n|\d \vb \d n\,,
\end{equation}
where the \emph{pre-collisional velocities} are defined through the relation
\eqref{'v'vb} and the Jacobian $J_e(\cdot)$ is given by $J_e(r)=re'(r)+e(r) \leq e(r)$ for any $r\geq 0$.  Introduce the sets
\begin{align*}
A_0:=\big\{(v,\vb) \in \mathbb{R}^6\,;\,|v|^2 &+ |\vb|^2 \leq \tfrac{3}{4}\big\}\,,\quad \text{and}\\
A:=\big\{(v,\vb,n) \in \mathbb{R}^6 \times \mathbb{S}^2\,;\, |'v|^{2}+&|'\vb|^{2}\leq 1\big\}
\subset\big\{ \,|'v|^{2}\leq1\big\}\cap\big \{  |'\vb|^{2}\leq 1\big\}.
\end{align*}
Thanks to \eqref{energ} we have
\begin{align*}
A&=\Big\{(v,\vb,n) \in \mathbb{R}^6 \times \mathbb{S}^2\,;\,\,;\,|v|^{2}+|v_{*}|^{2}\leq 1-\tfrac{1-'\!\!e^{2}}{2}\,('\!u\cdot n)^{2}\Big\}\\
&=\Big\{(v,\vb,n) \in \mathbb{R}^6 \times \mathbb{S}^2\,;\,|v|^{2}+|v_{*}|^{2}\leq 1-\tfrac{1-'\!\!e^{2}}{2'\!\!e^{2}}\,(u\cdot n)^{2}\Big\},
\end{align*}
where we adopted the short-hand notations $'\!\!e:=e(|'\!u \cdot n|)$.  For any $(v,\vb) \in A_0$, set
\begin{equation*}
B(v,\vb):=\big\{ n  \in  \mathbb{S}^2\,;\, |\hat{u}\cdot n|\leq\, '\!\!e/2\big\}.
\end{equation*}
It is not difficult to check that, if $(v,\vb) \in A_0$ and $n \in B(v,\vb)$ then
\begin{equation*}
\frac{1-'\!\!e^{2}}{2'\!\!e^{2}}\,(u\cdot n)^{2}\leq\frac{1-'\!\!e^{2}}{8}\;|u|^{2}\leq\frac{1}{4}.
\end{equation*}
and clearly  $(v,\vb,n) \in A$.  Now, for any $(v,\vb) \in A_0$, one computes
\begin{equation*}
I_e(v,\vb):=\int_{B(v,\vb)} \frac{1}{'e\,'J_e}\,|u \cdot n|\d n.
\end{equation*}
Noticing that $\frac{1}{'e}\,|u\cdot n|=|'\!u \cdot n|=\vartheta_e^{-1}(|u\cdot n|)$ one can write $B(v,\vb)=\big\{ n  \in  \mathbb{S}^2\,;\, \vartheta_e^{-1}(|u\cdot n|)\leq\, |u|/2\big\}$. Then,
\begin{align*}
I_e(v,\vb) &=  \frac{2|\mathbb{S}^{1}|}{|u|}\int^{|u|}_{0}\mathbf{1}_{\{x\leq\vartheta_e(|u|/2)\}}\;\frac{\vartheta_e^{-1}(x)}{J_e(\vartheta_e^{-1}(x))}\d x\\
&= \frac{2|\mathbb{S}^{1}|}{|u|}\int^{\vartheta_e(|u|/2)}_{0}\frac{\vartheta_e^{-1}(x)}{J_e(\vartheta_e^{-1}(x))}\d x.\end{align*}
Recalling that $J_e(r)=\tfrac{\d}{\d r}\vartheta_e(r)>0$, one gets
\begin{equation*}
I_e(v,\vb)= \frac{2|\mathbb{S}^{1}|}{|u|}\int^{|u|/2}_{0}y \d y=\frac{|\mathbb{S}^{1}|}{4}|u|.
\end{equation*}
Finally, since $\big\{|v|\leq \sqrt{3/8}\big\}\cap \big\{|\vb|\leq \sqrt{3/8}\big\}\subset A_0$ one sees that, for any $v \in \mathds{B}(0, \sqrt{3/8})$
\begin{equation*}\begin{split}
F_{e}(v)&=\int_{\mathbb{R}^3}\mathbf{1}_{A_0}(v,\vb)I_e(v,\vb)\d \vb \geq \frac{|\mathbb{S}^{1}|}{4} \int_{\mathds{B}(0,\sqrt{3/8})}|u|\d\vb\\
&\geq \frac{|\mathbb{S}^1|}{4} \int_{\mathds{B}(0,\sqrt{3/8})}\big|\,|v|-|\vb|\big|\d\vb=\frac{|\mathbb{S}^1||\mathbb{S}^2|}{4}
\int_0^{\sqrt{3/8}}\varrho^2 \big|\,|v|-\varrho\big|\d\varrho.\end{split}
\end{equation*}
Finally, one observes that
\begin{equation*}
\inf_{v \in \mathds{B}(0,\sqrt{3/8})}\int_0^{\sqrt{3/8}}\varrho^2 \big|\,|v|-\varrho\big|\d\varrho=\frac{9}{256}\left(1-\frac{1}{\sqrt[3]{2}}\right) > 0
\end{equation*}
from which \eqref{eqlowvel} follows.
\end{proof}
\noindent
Let us continue establishing some technical lemmas about the collision maps.  Let $r \in (1/2,1)$ and $v \in \mathbb{R}^3$ be fixed such that $|v|=\ell_e(r)$.  Introduce the pre-collision mapping
\begin{equation*}
\mathbf{'\Xi}\::\:(\vb,n) \in \mathbb{R}^3 \times \mathbb{S}^2 \longmapsto \mathbf{'\Xi}(\vb,n)=('\!\!v,'\!\!\vb)
\end{equation*}
with $('\!\!v,'\!\!\vb)$ given by \eqref{'v'vb}.  Introduce also the post-collision map
\begin{equation*}
\mathbf{\Xi}\::\: ('\!\!v,'\!\!\vb,n) \in \mathbb{R}^3 \times \mathbb{S}^2 \longmapsto \mathbf{\Xi}('\!\!v,'\!\!\vb,n)=('\!\!\vb+\beta_e(|'\!u \cdot n|)('\!u \cdot n) n,n).
\end{equation*}
Define also the sets
\begin{align*}
\mathbf{'\!\Omega} =\big\{('\!\!v,'\!\!\vb,n) \in \mathbb{R}^3 \times &\mathbb{S}^2\,;\, |'\!v|=|'\!v_{*}|=r,\;\; '\!v\cdot '\!v_{*}=0,\;\; '\!v\cdot n=0 \big\}\,, \quad \text{and}\\
\mathbf{\Omega} =\big\{(\vb,n) \in \mathbb{R}^3 &\times \mathbb{S}^2\,;\,|\vb|=R,\;\; \widehat{v}\cdot \widehat{v}_\star=\cos \alpha,\;\; \widehat{u}\cdot n=\cos \theta\big\},
\end{align*}
where $R:=r(1-\beta_e(r))$ and
\begin{equation*}
\cos \theta=\cos \theta(r)=\tfrac{\vartheta_e(r)}{\sqrt{r^{2}+\vartheta_e(r)^{2}}}\in[0,\tfrac{1}{\sqrt{2}}], \quad \cos \alpha=\cos \alpha(r)=\tfrac{\beta(r)}{\sqrt{1+\beta(r)^{2}}} \in [\tfrac{1}{\sqrt{5}},\tfrac{1}{\sqrt{2}}]
\end{equation*}
are fixed (depending on $v$ through $r$). It is easy to check that
\begin{equation*}
\mathbf{'\Xi}(\mathbf{\Omega} )= \mathbf{'\!\Omega}  \qquad \text{ and } \qquad \mathbf{\Xi}('\!\mathbf{\Omega} )=\mathbf{\Omega} .
\end{equation*}
For any $\varepsilon \in(0,1)$ define also the set
\begin{align}\label{def:omegaeps}
\mathbf{\Omega}_{\varepsilon} =\big\{(\vb,n) \in \mathbb{R}^3 \times \mathbb{S}^2\,;\, \big| |\vb|-R \big| < &\varepsilon^{2},\nonumber \\
\big| \widehat{v}\cdot \widehat{v}_\star - \cos \alpha \big| < & \varepsilon^{2},\;  \big |\widehat{u}\cdot n - \cos \theta  \big| < \varepsilon,\;  \widehat{u}\cdot n\geq0\big\}.
\end{align}
With these definitions one has the following technical lemma.
\begin{lemme}\label{geom}
Let $r \in (1/2,1)$ and $v \in \mathbb{R}^3$ such that $|v|=\ell_e(r)$ and $\varepsilon \in (0,1)$ be fixed. For any $(\vb,n) \in \mathbf{\Omega}_\varepsilon$ one has
\begin{equation}\label{uuo}
\Big| |v-\vb|- \frac{\vartheta_e(r)}{\cos \theta(r)}\Big| \leq 2\varepsilon^2.\end{equation}
Moreover, there exists an explicit $\varepsilon_0 \in (0,1)$ (independent of $r$) such that, for any $\varepsilon \in (0,\varepsilon_0)$
\begin{equation}\label{costheta}
|v-\vb|\left(\cos \theta + \varepsilon\right) \geq \vartheta_e(r)  \qquad \forall \vb \in \mathbf{\Omega}_\varepsilon.\end{equation}
Finally, if $r+4\varepsilon \leq 1$ then for any $(\vb,n) \in \mathbf{\Omega}_\varepsilon$
\begin{equation}\label{estimate'v}
r\,\Big|\,|'\!v|-r \big| \leq \big| r^{2}-|'\!v|^{2} \Big| \leq 4\varepsilon\left(\sqrt{5}+3\right)\,\vartheta_e^{-1}(1)\,\mathrm{Lip}_{[0,1]}(\vartheta_e^{-1}),
\end{equation}
Estimate \eqref{estimate'v} also holds if $|'v|$ is replaced by $|'\vb|$.
\end{lemme}
\begin{proof} Notice that
\begin{equation*}
|v-w|=\frac{\vartheta_e(r)}{\cos \theta}=\sqrt{r^2+\vartheta_e^2(r)}=r\sqrt{1+e^2(r)} \qquad \forall w \in \mathbf{\Omega}.
\end{equation*}
Therefore, for a given $\vb \in \mathbf{\Omega}_\varepsilon$
\begin{equation*}
\Big|\,|v-\vb|-\frac{\vartheta_e(r)}{\cos \theta}\,\Big|=\big|\,|v-\vb|-|v-w|\,\big| \leq |\vb-w| \qquad \forall w \in \mathbf{\Omega}
\end{equation*}
Thus, to estimate the left-hand side, it is enough to compute $|\vb-w|$ for a well chosen $w \in \mathbf{\Omega}$.  Choose $w \in \mathbf{\Omega}$ that lies in the plane $\mathcal{P}(v,\vb)$ determined by $v$ and $\vb$.  The intersection of $\mathcal{P}(v,\vb)$ with $\mathbf{\Omega}_\varepsilon$ is included in the symmetric cone with aperture angle given by
\begin{equation*}
\gamma_\varepsilon=\max\left\{\arccos\left(\cos\alpha(r) +\varepsilon^2\right) -\alpha(r), \alpha(r) -\arccos\left(\cos\alpha(r) -\varepsilon^2\right) \right\}.
\end{equation*}
Then $\gamma_\varepsilon \leq \tfrac{\varepsilon^2}{\sin\alpha(r)}$ and since $\sin\alpha(r)\in \left[\tfrac{1}{\sqrt{2}},\tfrac{2}{\sqrt{5}}\right]$, one concludes that $\gamma_\varepsilon \leq \sqrt{2}\varepsilon^2$.  Therefore,
\begin{equation*}
|\vb-w| \leq \big|\,|\vb|-|w|\,\big| + |w|\,|\widehat{v}_\star-\widehat{w}| \leq \big|\,|\vb|-|w|\,\big| + R\gamma_\varepsilon \leq \varepsilon^2 + R\sqrt{2}\varepsilon^2
\end{equation*}
which proves \eqref{uuo} since $R \leq 1/2$.
Moreover, one always has
\begin{equation}\label{lowu}
|v-\vb| \geq \big|\,\ell_e(r)-|\vb|\,\big| \geq\ell_e(r) -\varepsilon^2 -R \geq \frac{r}{2}-\varepsilon^2.
\end{equation}
Then, writing
\begin{equation*}
|v-\vb|\left(\cos \theta + \varepsilon\right)=\vartheta_e(r) + \varepsilon\,|v-\vb| \big(1 + F_\varepsilon(|v-\vb|)\big)
\end{equation*}
with $F_\varepsilon(|v-\vb|)=\tfrac{\cos\theta}{\varepsilon|v-\vb|}\big(|v-\vb|  -\vartheta_e(r)\big)$, one sees from \eqref{uuo} and \eqref{lowu} that
\begin{equation*}
\big|F_\varepsilon(|v-\vb|)\big| \leq  \dfrac{2\varepsilon}{r/2-\varepsilon^2} \leq \dfrac{2\varepsilon}{1/4-\varepsilon^2} < 1/2 \quad  \text{if} \quad \varepsilon < -2+\tfrac{\sqrt{17}}{2}.
\end{equation*}
We obtain \eqref{costheta} by setting $\varepsilon_0= -2+\tfrac{\sqrt{17}}{2} > 1/10$.
Let us prove now \eqref{estimate'v} recalling that given $(\vb,n) \in \mathbf{\Omega}_\varepsilon$ one has
\begin{equation*}
'\!v=v-\xi_e\big(|u\cdot n|\big)n,  \quad \text{with} \quad \;\xi_e\big(|u\cdot n|\big)=\tfrac{1}{2}\big( \vartheta_e^{-1}(|u\cdot n|)+|u\cdot n| \big).
\end{equation*}
Notice that $\vartheta_e^{-1}$ is a Lipschitz nondecreasing function over $[0,1]$ and therefore the same holds for $\xi_e$ with
\begin{equation}\label{lip}
\mathrm{Lip}_{[0,1]}(\xi_e) \leq \mathrm{Lip}_{[0,1]}(\vartheta_e^{-1}).
\end{equation}
One can check  that for $(\vb,n)$ in $\mathbf{\Omega}_\varepsilon$
\begin{equation*}
\big|\vartheta_e(r) - |u \cdot n|\,\big| \leq \bigg|\,|u|-\frac{\vartheta_e(r)}{\cos \theta}\,\bigg| + \frac{\vartheta_e(r)}{\cos \theta}\big| \widehat{u}\cdot n - \cos \theta\big| \leq 4\varepsilon, \qquad u=v-\vb.\end{equation*}
Additionally, for any $(w,n) \in \mathbf{\Omega}$ denote for simplicity $\mathbf{'\Xi}(w,n)=(w_1,w_2)$.  Therefore, $|w_1|=r$ and
\begin{equation*}
\big||'v|^2 - r^2\big|=\big|\langle 'v-w_1, 'v + w_1\rangle \big| \leq \left(|'v|+r\right)|'v-w_1|.
\end{equation*}
Since $(w,n) \in \mathbf{\Omega}$ one checks that $w_1=v-\xi_e\big(\vartheta_e(r)\big)n$.  Consequently,
\begin{equation*}
\big||'v|^2 - r^2\big| \leq \left(|'v|+r\right)\left|\xi_e\big(|u \cdot n|\big) - \xi_e\big(\vartheta_e(r)\big)\right|.
\end{equation*}
Noticing that
\begin{equation*}
|'v| \leq |v| +  \xi_e\big(|u\cdot n|\big) =\ell_e(r) + \xi_e\big(|u\cdot n|\big) \leq \sqrt{5} r\beta_e(r)+ \xi_e\big(|u\cdot n|\big) \quad \text{and} \quad r\beta_e(r)=\xi_e(\vartheta_e(r)),
\end{equation*}
it follows
\begin{multline*}
\big||'v|^2 - r^2\big| \leq \left(\sqrt{5}\,\xi_e\big(\vartheta_e(r)\big)+ r + \xi_e\big(|u \cdot n|\big)\right)\bigg|\xi_e\big(|u \cdot n|\big) - \xi_e\big(\vartheta_e(r)\big)\bigg|\\
\leq \left((\sqrt{5}+2)\,\xi_e\big(\vartheta_e(r)\big) + \xi_e\big(|u\cdot n|\big) \right) \bigg|\xi_e\big(|u \cdot n|\big) - \xi_e\big(\vartheta_e(r)\big)\bigg|.
\end{multline*}
Choosing  $\varepsilon$ is such that $r+4\varepsilon \leq 1$ we get $|u\cdot n| \leq 1$ and
\begin{equation*}
\big||'v|^2 - r^2\big| \leq 4\varepsilon  \left(\sqrt{5}+3\right)\xi_e(1) \mathrm{Lip}_{[0,1]}(\xi_e).
\end{equation*}
Since $\xi_e(1) \leq \vartheta_e^{-1}(1)$ we conclude \eqref{estimate'v} for $'\!\!v$ thanks to \eqref{lip}.  Using $w_2$ in the place of $w_1$ we obtain it for $'\vb=v+\xi_e(|u\cdot n|)n$.
\end{proof}
\begin{lemme}\label{lemEpsi}
Let $r \in (1/2,1)$ and $v \in \mathbb{R}^3$ with $|v|=\ell_e(r)$.  Assume that
\begin{equation}\label{estivare}
\dfrac{1-r}{ 4\left(\sqrt{5}+3\right)\,\vartheta_e^{-1}(1)\,\mathrm{Lip}_{[0,1]}(\vartheta_e^{-1})} >\varepsilon.
\end{equation}
Then, $F_e(v) \geq \dfrac{8\pi^2}{3}\,r\,\varepsilon^9$.
\end{lemme}
\begin{proof} Fix $r \in (1/2,1)$ and $v \in \mathbb{R}^3$ with $|v|=\ell_e(r)$ and $\varepsilon$ satisfying \eqref{estivare}.  Since $\varepsilon_0 > 1/10$ we have readily that $\varepsilon < \varepsilon_0$ (where $\varepsilon_0$ is the parameter appearing in Lemma \ref{geom}).  As in Lemma \ref{Fesmall}, we use \eqref{strong} to obtain
\begin{equation}\label{eq1lv}
F_{e}(v)\geq \int_{\mathbf{\Omega}_{\varepsilon}(v)}\frac{1}{e\big(|'u \cdot n|\big)J_e\big(|'u\cdot n|\big)}\mathbf{1}_{\mathds{B}(0,1)}('v)\,\mathbf{1}_{\mathds{B}(0,1)}('\vb) |u\cdot n|\d\vb \d n.
\end{equation}
Since $|'v|$ satisfies \eqref{estimate'v} and $r>1/2$, it follows that
\begin{equation*}
|'v| \leq r + \big|\,|'v|-r\,\big| \leq r + 8\varepsilon\left(\sqrt{5}+3\right)\,\vartheta_e^{-1}(1)\,\mathrm{Lip}_{[0,1]}(\vartheta_e^{-1}) < 1,
\end{equation*}
and the same being true for $|'\vb|$.   Consequently, from \eqref{eq1lv} one concludes
\begin{equation}\label{eq2lv}
 F_e(v) \geq \int_{\mathbf{\Omega}_{\varepsilon}(v)}\frac{|u \cdot n|}{e\big(|'u \cdot n|\big)J_e\big(|'u\cdot n|\big)}\d\vb \d n.
\end{equation}
Recall the definition \eqref{def:omegaeps} of $\mathbf{\Omega}_\varepsilon$ and let $\vb \in \mathbb{R}^3$ be given with the property
$\big| |\vb|-R \big| < \varepsilon^{2},$ and $\big| \widehat{v}\cdot \widehat{v}_\star - \cos \alpha \big| < \varepsilon^{2}$ and set
\begin{equation*}
G_e(v,\vb)=\int_{\mathbb{S}^2_+} \mathbf{1}_{\{|\widehat{u}\cdot n - \cos \theta | < \varepsilon\}}(n)\frac{|u \cdot n|}{e\big(|'u \cdot n|\big)J_e\big(|'u\cdot n|\big)}\d n
\end{equation*}
where $\mathbb{S}^2_+=\{n \in \mathbb{S}^2\,,\,\widehat{u} \cdot n \geq 0\}$.  One can compute $G_e(v,\vb)$ as in Lemma \ref{Fesmall} and obtain
\begin{equation*}
G_e(v,\vb)=\frac{|\mathbb{S}^{1}|}{|u|}\int^{\vartheta_e^{-1}(|u|(\cos \theta+\varepsilon))}_{\vartheta_e^{-1}(0 \wedge |u|(\cos \theta-\varepsilon))} x \,\d x.
\end{equation*}
Then, using \eqref{costheta} and the fact that $\vartheta_e^{-1}$ is non-negative and increasing it follows that
\begin{equation*}
\left(\vartheta_e^{-1}\big(|u|(\cos \theta+\varepsilon)\big) + \vartheta_e^{-1}\big(0 \wedge |u|(\cos \theta-\varepsilon)\big)\right) \geq \vartheta_e^{-1}\big(\vartheta_e(r)\big)=r.
\end{equation*}
Since $\vartheta_e^{-1}(0)=0$ and $\dfrac{\d}{\d s}\vartheta_e^{-1}(s) \geq 1$ for any $s$ we conclude that
\begin{equation*}
\left(\vartheta_e^{-1}\big(|u|(\cos \theta+\varepsilon)\big) -\vartheta_e^{-1}\big(0 \wedge |u|(\cos \theta-\varepsilon)\big)\right) \geq \big(2\varepsilon|u| \wedge |u|(\cos \theta + \varepsilon)\big) \geq \varepsilon|u|.
\end{equation*}
Accordingly,
\begin{equation*}
G_e(v,\vb) \geq \frac{|\mathbb{S}^1|}{2}\, r\, \varepsilon.
\end{equation*}
Integrating $G_e(v,\vb)$ with respect to $\vb$, we obtain from \eqref{eq2lv} and polar coordinates that
\begin{equation*}\begin{split}
F_e(v) &\geq \int_{\mathbb{R}^3} \mathbf{1}_{\{| |\vb|-R \big| < \varepsilon^{2},\,| \widehat{v}\cdot \widehat{v}_\star - \cos \alpha| < \varepsilon^{2}\}}G_e(v,\vb)\d \vb\\
&\geq \frac{|\mathbb{S}^1|^2}{2} r \varepsilon \int_{R-\varepsilon^2}^{R+\varepsilon^2}\varrho^2 \d\varrho \int_{\cos \alpha-\varepsilon^2}^{\cos \alpha+\varepsilon^2} \d s=\dfrac{4\pi^2\,r\,\varepsilon^3}{3}\left((R+\varepsilon^2)^3-(R-\varepsilon^2)^3\right),
\end{split}\end{equation*}
which yields the conclusion.
\end{proof}
\noindent
We are now in position of handling large velocities and prove Proposition \ref{spread}.
\begin{proof}[Proof of Proposition \ref{spread}] We assume without loss of generality that $v_0=0$ and we first consider the case $\delta=1$.  Let us fix $\chi>0$ and $v \in \mathbb{R}^3$ with
\begin{equation*}
|v|=(1-\chi)\ell_{e}(1),
\end{equation*}
and set $r \in (0,1)$ such that $\ell_{e}(r)=(1-\chi)\ell_{e}(1)$.  Since the mapping $\ell_{e}(\cdot)$ is Lipschitz with norm $\sqrt{2}$ we have
\begin{equation*}
\tfrac{\sqrt{5}}{2}\chi\leq\ell_{e}(1)\chi=\ell_{e}(1)-\ell_{e}(r)\leq \sqrt{2}(1-r).
\end{equation*}
Choose $\varepsilon$ as
\begin{equation*}
\varepsilon=\dfrac{\sqrt{2}\,\chi}{ 2\sqrt{5}\left(\sqrt{5}+3\right)\,\vartheta_e^{-1}(1)\,\mathrm{Lip}_{[0,1]}(\vartheta_e^{-1})}.
\end{equation*}
Such $\varepsilon$ satisfies \eqref{estivare} and $r \in (1/2,1)$ provided $\chi \in (0,\sqrt{2/5})$.  Therefore, according to Lemma \ref{lemEpsi}, one concludes that
\begin{equation*}
F_e(v) \geq \dfrac{16\pi^2}{3}\,r\,\varepsilon^9 \geq \dfrac{8\pi^2}{3}\varepsilon^9
\end{equation*}
from which we obtain the existence of a universal constant $C_0 >0$ such that
\begin{equation*}
F_e(v) \geq C_0 \left(\dfrac{\chi}{\vartheta_e^{-1}(1)\,\mathrm{Lip}_{[0,1]}(\vartheta_e^{-1})}\right)^9.
\end{equation*}
Furthermore, since $\ell_e(1) \leq \sqrt{2}$ for any $e(\cdot)$ the inequality
$(1-\sqrt{2/5})\ell_{e}(1) < \sqrt{3/8}$ always holds. Therefore, if $\chi \in (\sqrt{2/5},1)$ then $v \in \mathds{B}(0,\sqrt{3/8})$ and Lemma \ref{Fesmall} holds true. Thus, in any of the cases, we obtain the existence of a universal constant $\kappa > 0$ such that
\begin{equation}\label{finalest}
F_{e}(v)\geq \kappa \left(\dfrac{\chi}{\vartheta_e^{-1}(1)\,\mathrm{Lip}_{[0,1]}(\vartheta_e^{-1})}\right)^9 \,\mathbf{1}_{\mathds{B}(0,(1-\chi)\ell_{e}(1))},\forall \chi\in(0,1), \forall v \in \mathbb{R}^3.
\end{equation}
For the general case $\delta>0$, fix $v\in\mathbb{R}^{3}$, $\delta>0$ and define $w=\delta^{-1}v$.  Since
\begin{equation*}
\Q^+_e\big(\mathbf{1}_{\mathds{B}(0,\delta)}\,;\,\mathbf{1}_{\mathds{B}(0,\delta)}\big)(v)=\delta^4\Q^+_{e_\delta}\big(\mathbf{1}_{\mathds{B}(0,1)}\,;\,\mathbf{1}_{\mathds{B}(0,1)}\big)(w)
\end{equation*}
we deduce from \eqref{finalest} that
\begin{align*}
\Q^+_e\left(\mathbf{1}_{\mathds{B}(0,\delta)}\,;\,\mathbf{1}_{\mathds{B}(0,\delta)}\right)(v) &\geq \kappa \delta^4 \left(\dfrac{\chi}{\vartheta_{e_\delta}^{-1}(1)\,\mathrm{Lip}_{[0,1]}(\vartheta_{e_\delta}^{-1})}\right)^9 \,\mathbf{1}_{\mathds{B}(0,(1-\chi)\ell_{e_\delta}(1))}(w)\\
&=\kappa \delta^{13} \left(\frac{\chi}{\vartheta_e^{-1}(\delta)\,\mathrm{Lip}_{[0,\delta]}(\vartheta_e^{-1})}\right)^9 \mathbf{1}_{\mathds{B}(0,(1-\chi)\ell_{e}(\delta))}(v) \end{align*}
where we simply used the fact that
$$\ell_{e_\delta}(1)=\frac{1}{\delta}\ell_e(\delta) \qquad \text{ and } \quad \vartheta_{e_\delta}^{-1}(s)=\frac{1}{\delta}\,\vartheta^{-1}_{e}(\delta s)$$
for any $s \geq 0$ and any $\delta >0.$
\end{proof}
\subsection{Uniform spreading properties of the iterated $\Q_e^+$}
The main objective of this subsection is to prove Lemma \ref{l1} which gives the existence of a uniform lower bound for the iterated Boltzmann collision operator in terms of the conserved quantities mass, energy and the propagating quantity $L^{p}$-norm.  The approach used for this proof is taken from \cite{ada, MiMo2}.   Let us start introducing some useful relations and definitions.
\begin{lemme}\label{lemmB1}
Define $\beta_e(r)=\frac{1+e(r)}{2}$ and the mapping
\begin{equation*}
\eta_e\::\:r \in \mathbb{R}^+ \longmapsto r\beta_e(r).
\end{equation*}
Then,  $\eta_e(\cdot)$ is strictly increasing and differentiable with
\begin{equation*}
\frac{r}{2} \leq \eta_e(r) \leq r\,; \quad \frac{1}{2} \leq  \eta_e'(r) \leq \frac{\eta_e(r)}{r} \quad \text{for any} \quad  r > 0 \quad \text{and} \quad \eta_e'(0)=1.
\end{equation*}
Equivalently, the inverse mapping $\alpha_e(\cdot)$ of $\eta_e(\cdot)$ satisfies
\begin{equation*}
r \leq \alpha_e(r) \leq 2r\,; \quad  \frac{\alpha_e(r)}{r} \leq  \alpha_e'(r) \leq 2 \quad \text{for any} \quad  r > 0 \quad \text{and} \quad \alpha_e'(0)=1.
\end{equation*}
\end{lemme}
\begin{proof}
Refer to \cite[Lemma A.1]{AloLo3}.
\end{proof}
\begin{propo}[\textbf{Carleman representation for hard spheres}]\label{carle} Let $e(\cdot)$ be of class $\mathcal{R}_0$. Then, for any velocity distributions $f,g$  one has
\begin{equation}\label{eq:Carle}\Q_{e}^+\big(f,g\big)(v)=\tfrac{2}{\pi}\IR f(w)\Delta_e\big(|v-w|\big)\d w \int_{\left(v-w\right)^\perp} g(\chi_{v,w}^e + z)\d \pi(z)\end{equation}
where for any $v,w \in \mathbb{R}^3$, $\d\pi(z)$ is the Lebesgue measure over the hyperplane $(v-w)^{\perp}$,
\begin{equation*}
\chi_{v,w}^e=w+\alpha_e\big(|v-w|\big)\dfrac{v-w }{|v-w|},
\end{equation*}
and
\begin{equation*}
\Delta_e(r)= \dfrac{\alpha_e(r)}{r^2 \left(1+\vartheta'_e\big(\alpha_e(r)\big)\right)}, \qquad r > 0.
\end{equation*}
\end{propo}
\begin{proof}
Refer to \cite[Corollary 4.2]{AloLo1} and \cite[Lemma 4.1]{AloLo1}.
\end{proof}
\begin{nb}\label{Deltae}
Under the assumptions on $e(\cdot)$
\begin{equation*}
0 \leq \vartheta_e'(\rho)=\rho\,e'(\rho)+e(\rho) \leq e(\rho) < 1.
\end{equation*}
Hence, according to Lemma \ref{lemmB1},
\begin{equation*}
2 \geq r\Delta_e(r) \geq \dfrac{\alpha_e(r)}{2r} \geq \dfrac{1}{2} \qquad \forall r >0.
\end{equation*}
\end{nb}
\begin{defi}\label{hyperplane}
For any $v,\,\vb$ let
\begin{equation*}
\mathbf{P}_{v,\vb}^e=\chi_{v,\vb}^e+ (v-\vb)^\perp
\end{equation*}
denote the hyperplane passing through $\chi_{v,\vb}^e$ and orthogonal to $v-\vb$.\\

\noindent
Additionally, let $\mathcal{I}_{v,v_\star}^e$ stand  for the set of all possible post-collisional velocity $v'$
\begin{equation*}
\mathcal{I}_{v,v_\star}^e=\mathbf{\Phi}_{v,\vb}^e\left(\mathbb{S}^2\right)
\end{equation*}
where $\mathbf{\Phi}_{v,\vb}^e\::\:\mathbb{S}^2 \to \mathbb{R}^3$ is the post-collisional map defined by
\begin{equation}\label{postmap}
\mathbf{\Phi}_{v,\vb}^e(\sigma)=\v-\beta_e\left(|u| \sqrt{\tfrac{1-\widehat{u} \cdot \sigma}{2}}\right)\frac{u-|u|\sigma}{2}, \qquad \sigma \in \mathbb{S}^2, \;\; u=v-v_\star.
\end{equation}
\end{defi}
\begin{lemme}\label{l0}
Fix $p\in(1,\infty]$ and assume that $g \geq 0$ is such that
\begin{equation}\label{itpe1}
\int_{\mathbb{R}^{3}}g(v)\d v=m_0,\quad
\int_{\mathbb{R}^{3}}g(v)|v|^{2}\d v \leq m_2<\infty, \quad \|g\|_p^p:=\int_{\mathbb{R}^3} |g(v)|^p \d v < \infty
\end{equation}
For any $R > \sqrt{\frac{2m_2}{m_0}}$ and any $\kappa > 0$, there exist positive  $r$ and $\eta$ depending only on $m_0$, $m_2$, $\|g\|_p$, $R$ and $\kappa$, and velocities $v_1,\,v_2$ satisfying
\begin{enumerate}[(i)]
\item $|v_{i}|\leq \sqrt{3}\,R \quad i=1,2$.
\item $|v_{1}-v_{2}|\geq \kappa\,r$,
\item $\displaystyle \int_{\mathds{B}(v_{i},r)}g(v)\d v\geq\eta \quad i=1,2$.
\item Moreover, choosing $\kappa >0$ large enough (and therefore $r >0$ small enough),
\begin{equation}\label{delta}
\int_{\mathbb{S}^2} \delta_0\left(\left(\mathbf{\Phi}_{w_2,w_1}^e(\sigma)-\chi_{w_4,w_3}^e\right)\cdot \dfrac{w_4-w_3}{|w_4-w_3|}\right)\d\sigma \geq \dfrac{\pi}{\sqrt{3}R}.
\end{equation}
for any $w_{i}\in \mathds{B}(v_{i},r)$ ($1\leq i\leq4$) where $v_3=v_1$ and $v_4=\frac{v_1+v_2}{2}$.
\end{enumerate}
\end{lemme}
\begin{proof}
Let $R > \sqrt{\frac{2m_2}{m_0}}$ and denote by $C_R:=[-R,R]^{3}$ the  cube with center in the origin and length $2R>0$.  Thus,
\begin{equation}\label{l0e1}
\int_{C_R}g(v)\d v\geq\int_{\{|v|\leq R\}}g(v)\d v=m_0-\int_{\{|v|\geq R\}}g(v)\d v\geq m_0-\frac{m_2}{R^{2}} \geq \frac{m_0}{2}.
\end{equation}
Let $\lambda >0 $ and $r>0$ to be chosen later on such that $R/r\in\mathbb{N}$.  We define a family $\left(C_{i}\right)^{I}_{i=1}$ of $I:=(2R/r)^{3}$ cubes of length $r>0$ covering $C_R$.  For any $1\leq i\leq I$, define then $K_{i}$ as  the cube with same center of $C_{i}$ and length $\lambda\,r$. Clearly, we can choose $r >0$ small enough in order that
\begin{equation}\label{l0e2}
\int_{K_i}g(v)\d v\leq\big|K_{i}\big|^{1/p'}\,\|g\|_{p}\leq  (\lambda \,r)^{3/p'}\,\|g\|_{p}\leq\frac{m_0}{4} \qquad \forall i=1,2,\ldots, I.
\end{equation}
Choose $i_1$ such that the mass of $g$ is maximal in $C_{i}$ for $i=i_1$, that is
\begin{equation*}
\int_{C_{i_1}} g(v)\d v =\max_{1\leq i \leq I} \int_{C_i}g(v)\d v.
\end{equation*}
Denoting by $v_1$ the center of $C_{i_1}$ one concludes using \eqref{l0e1} that
\begin{equation}\label{l0e3}
\int_{\mathds{B}(v_{1},\frac{\sqrt{3}}{2}r)}g(v)\d v\geq\int_{C_{i_1}}g(v)\d v\geq \frac{m_0}{2I}=\frac{m_0}{2}\left(\frac{r}{2R}\right)^{3}=:2\eta.
\end{equation}
Moreover, using \eqref{l0e2} we conclude that
\begin{equation*}
\int_{C_R \setminus K_{i_1}}g(v)\d v\geq\frac{m_0}{4}.
\end{equation*}
Thus, the previous argument shows the existence of a cube $C_{i_2}$ in $C_R\setminus K_{i_1}$ such that
\begin{equation}\label{l0e4}
\int_{\mathds{B}(v_{2},\frac{\sqrt{3}}{2}r)}g(v)\d v\geq\int_{C_{i_2}}g(v)\d v\geq \frac{m_0}{4I}=\eta,
\end{equation}
where $v_2$ is the center of $C_{i_2}$.  Since $\text{dist}\big(C_{i_1},C_{i_2})\geq \tfrac{\lambda+1}{2}\,r $, it is possible to choose $\lambda > 2\kappa -1$ so that estimates \eqref{l0e3} and \eqref{l0e4} yield $(i)-(ii)-(iii)$ for that choice of $\lambda$ and $r >0$.\\

\noindent
Let us now prove $(iv)$.   For any $w_i \in \mathbb{R}^3$, define
\begin{equation*}
D(w_1,w_2,w_3,w_4):=\int_{\mathbb{S}^2} \delta_0\left(\left( \mathbf{\Phi}_{w_2,w_1}^e(\sigma)-\chi_{w_4,w_3}^e\right)\cdot \dfrac{w_4-w_3}{|w_4-w_3|}\right)\d\sigma.
\end{equation*}
We want to bound $D(w_1,w_2,w_3,w_4)$ uniformly on the balls $\prod_{i=1}^4 \mathds{B}(v_i,r)$ by some constant independent of the radius $r >0$.  For given $w_1,\ldots,w_4$, we set
\begin{equation*}
u=w_2-w_1, \quad \varpi=w_4-w_3, \quad \hat{u}=\frac{u}{|u|}, \quad \hat{\varpi}=\frac{\varpi}{|\varpi|}.
\end{equation*}
Then, for any $\sigma \in \mathbb{S}^2$
\begin{multline}\label{const}
\big( \mathbf{\Phi}_{w_2,w_1}^e(\sigma)-\chi_{w_4,w_3}^e\big) \cdot \dfrac{w_4-w_3}{|w_4-w_3|} =\left( \widehat{\varpi}\cdot( w_2-w_3) - \alpha_e(|\varpi|)\right)\\
-  \frac{|u|}{2}\beta_e\left(|u|\sqrt{\tfrac{1-\widehat{u}\cdot \sigma}{2}}\right)\big(\widehat{u}\cdot\widehat{\varpi} - \widehat{\varpi}\cdot \sigma\big).
\end{multline}
Set $A =A (\varpi,w_2,w_3)= \widehat{\varpi}\cdot\big( w_2-w_3\big) - \alpha_e\big(|\varpi|\big)$ and choose a frame of reference such that $\widehat{u}=(0,0,1)$ and $\widehat{\varpi}=(\sin \chi \sin \varphi,\, \sin \chi \cos\varphi,\, \cos \chi)$ for some $\varphi \in (0,2\pi)$ and $\chi \in (0,\pi)$.  Use spherical coordinates to compute
\begin{multline}\label{Dw1}
D(w_1,w_2,w_3,w_4)=\int_0^{2\pi}\d\phi \int_0^{\pi} \sin\theta\d\theta\\
\delta_0\left(A-\frac{|u|}{2}\beta_e\left(|u|\sqrt{\tfrac{1-\cos \theta}{2}}\right)\big(\cos \chi (1-\cos \theta) - \sin\theta \sin \chi\cos (\varphi- \phi)\big)\right).
\end{multline}
We first estimate $D$ in the centers $v_1,v_2,v_3,v_4$.  In this case
\begin{equation*}
u=v_2-v_1, \quad \varpi=\frac{u}{2} \quad \text{and} \quad  \widehat{\varpi}=\widehat{u}.
\end{equation*}
In particular, $\chi=0$ and  $A=|u|-\alpha_e\big(|u|/2\big)$.  In such case \eqref{Dw1} reads
\begin{equation}
\begin{split}\label{Dv1}
D(v_1,v_2,v_3,v_4)&={2\pi}\int_0^{\pi} \delta_0\left(A - \frac{|u|}{2 }\beta_e\left(|u|\sqrt{\tfrac{1-\cos \theta}{2}}\right)\big(1-\cos \theta\big)\right) \sin\theta\d\theta\\
&=\frac{8\pi}{|u|^2} \int_{0}^{|u|} \delta_0\left(A-\dfrac{s^2}{|u|}\beta_e(s)\right)s \d s
\end{split}
\end{equation}
where we performed the change of variables $s=|u|\sqrt{\frac{1-\cos\theta}{2}}$. Introduce then, using the notation of Lemma \ref{lemmB1}, $j_e(s)=s^2\beta_e(s)=s\eta_e(s)$ for any $s >0$.  Since $j_e$ is strictly increasing its inverse exists which we denote by $\gamma_e(\cdot)$. Performing in \eqref{Dv1} the change of variable $r=\frac{j_e(s)}{\xi |u|} \in \left(0,\eta_e(|u|)\right)$ we obtain
\begin{equation}\label{aidm}
D(v_1,v_2,v_3,v_4)=\dfrac{8  \pi}{|u|}\int_0^{ {\eta_e(|u|)} } \delta_0(A-r)\dfrac{\gamma_e(r |u|)}{j_e'\big(\gamma_e(r|u|)\big)}\d r
\end{equation}
where the derivative $j_e'$ is given by $j_e'(s)=s\eta_e'(s)+\eta_e(s)$.  According to Lemma \ref{lemmB1}, one notices that $A =|u|-\alpha_e\big(|u|/2\big) \in \big[0,\eta_e(|u|)\big]$ so that the integral \eqref{aidm} is non zero.  More precisely
\begin{equation*}
D(v_1,v_2,v_3,v_4)=\dfrac{8 \pi}{|u|}\dfrac{\gamma_e(A|u|)}{j_e'\left(\gamma_e(A |u|)\right)}\,.
\end{equation*}
Using Lemma \ref{lemmB1}, it follows that $j_e'(s) \geq 2\eta_e(s) \geq 2s$, thus
\begin{equation}\label{Dv2}
D(v_1,v_2,v_3,v_4) \geq \dfrac{4 \pi}{|u|}=\dfrac{4 \pi}{|v_2-v_1|} \geq \dfrac{2 \pi}{\sqrt{3}R}
\end{equation}
where we used point $(i)$.  We conclude that in the centers of the balls $D$ is uniformly bounded away from zero by some constant independent of $r$.  Turning back to the general case, using \eqref{Dw1} one obseves that $D(w_1,w_2,w_3,w_4)$ depends continuously on $|\varpi|,$ $w_2,w_3,$  $|u|$, $\varphi \in (0,2\pi)$ and $\chi \in (0,\pi)$.  Then it is possible to choose $\kappa >0$ large enough (recall that $|v_2-v_1| \geq \kappa r$) and $r >0$ small enough to have
\begin{equation*}
1-\cos \chi =\mathrm{O}(\kappa^{-2}) \quad\text{ and } \quad \sin \chi= \mathrm{O}(\kappa^{-1}) \qquad \text{ whenever } w_i \in \mathds{B}(v_i,r),\;\;1\leq i\leq 4.
\end{equation*}
Using uniform continuity
\begin{equation*}
D(w_1,w_2,w_3,w_4) \geq \frac{1}{2}D(v_1,v_2,v_3,v_4) \qquad \forall w_i \in \mathds{B}(v_i,r), \; \; 1\leq i \leq 4,
\end{equation*}
and we get the result thanks to \eqref{Dv2}.
\end{proof}
\begin{nb} The fact that $D(w_1,w_2,w_3,w_4)$ is non zero exactly means that $\mathbf{P}_{w_4,w_3}^e\cap \mathcal{I}^{e}_{w_2,w_1} \neq \varnothing$ for any $w_i \in \mathds{B}(v_i,r)$ ($1\leq i \leq 4$).  One can show as in \cite{MiMo3} that
\begin{equation*}
|\mathbf{P}_{w_4,w_3}^e \cap \mathcal{I}_{w_2,w_1}^e| \geq C r \qquad \forall w_i \in \mathds{B}(v_i,r) \;\; 1\leq i\leq 4
\end{equation*}
for arbitrarily large $C >0$.   However, it appears to us that such an estimate of the measure of
$\mathbf{P}_{w_4,w_3}^e \cap \mathcal{I}_{w_2,w_1}^e$ is not enough to estimate $D(w_1,w_2,w_3,w_4)$.
\end{nb}
\begin{nb} In the special case of constant coefficient of normal restitution $e(r)=e_0$, one sees from \eqref{const} that the integrand in $D(w_1,w_2,w_3,w_4)$ depends only on $\widehat{\varpi}\cdot \sigma$.  Choosing then a frame of reference such that $\widehat{\varpi}=(0,0,1)$ and $\widehat{u} \cdot \widehat{\varpi}=\cos \chi$
\begin{equation*}\begin{split}
D(w_1,w_2,w_3,w_4)&=2\pi \int_0^{\pi}  \delta_0\left(A -\frac{|u|(1+e_0)}{4}\left( \cos \chi - \cos \theta\right)\right)\sin\theta\d\theta\\
&=\frac{8\pi}{(1+e_0)|u|}\int_{- \frac{|u|(1+e_0) }{4}}^{\frac{|u|(1+e_0) }{4}}\delta_0\left(A -\frac{|u|(1+e_0)\cos \chi}{4}-s\right)\d s\\
&=\frac{8\pi}{(1+e_0)|u|} \geq \frac{2\pi}{(1+e_0)\sqrt{3}R}.
\end{split}\end{equation*}
The integral is non zero since $\mathbf{P}_{w_4,w_3}^e\cap  \mathcal{I}^{e}_{w_2,w_1} \neq \varnothing$.
\end{nb}
\begin{lemme}\label{l1}
Let  $f$, $g$ and $h$ be nonnegative and satisfying \eqref{itpe1}, and  assume that the centers  $v^{f}_{i}=v^{g}_{i}=v^{h}_{i}=v_{i}$ with $i=1,2$.  Then, there exist $r$ and $\eta_0$ depending only on upper bounds on $\E_f,\,\E_g,\,\E_h$ and $\max\{\|f\|_p,\|g\|_p,\|h\|_p\}$ and such that
\begin{equation*}
\Q^{+}_{e}\big(f,  \Q^{+}_{e}(g,h)\big) \geq \eta_0\,\mathbf{1}_{\mathds{B}\left(\tfrac{v_1+v_2}{2},\,r\right)}.
\end{equation*}
\end{lemme}
\begin{proof}
Let $f$, $g$ and $h$ nonnegative functions  satisfying \eqref{itpe1}  be fixed. For simplicity, we set $G:= \Q^+_{e}(g,h)$.  Then, from Carleman representation \eqref{eq:Carle}
 \begin{equation}\label{eq:it}
\Q^+_{e}\big(f,G\big)(v)=\dfrac{2}{\pi}\IR f(w)\Delta_{e}(|v-w|)\d w \int_{\left(v-w\right)^\perp} G(\chi_{v,w}^{e} + z)\d \pi(z).
\end{equation}
Let $v,\, w \in \mathbb{R}^3$ be fixed with $u=v-w$ and $u\neq 0$.  Define
\begin{equation*}
\mathcal{A}(v,w):=\int_{\left(v-w\right)^\perp}G (\chi_{v,w}^e + z)\d \pi(z)
=\int_{\mathbb{R}^3} G(z)\delta_0((z-\chi_{v,w}^e) \cdot \widehat{u})\d z
\end{equation*}
where we used the general identity, valid for any $x \in \mathbb{R}^3$, $x \neq 0$ and $F(\cdot)$
\begin{equation*}
\IR F(z)\delta(x \cdot z)\d z =\frac{1}{|x|} \int_{x^\perp} F(z)\d\pi(z).
\end{equation*}
Therefore, using that $G=\Q_e^+(g,h)$ we obtain
\begin{equation*}
\mathcal{A}(v,w)=\dfrac{1}{2 } \IRR g(w_2)h(w_1)|w_2-w_1|\d w_2 \d w_1 \int_{\mathbb{S}^2} \delta_0\big(\left( \mathbf{\Phi}_{w_2,w_1}^e(\sigma)-\chi_{v,w}^e\right)\cdot\widehat{u}\big)\d \sigma.
\end{equation*}
From Lemma \ref{l0} we conclude that if $v \in \mathds{B}(v_4,r) $ and $w \in \mathds{B}(v_3,r)$ then
\begin{equation*}\begin{split}
\mathcal{A}(v,w)
&\geq \dfrac{1}{2 } \int_{\mathds{B}(v_1,r)} h(w_1)\d w_1 \int_{\mathds{B}(v_2,r)}  g(w_2)|w_2-w_1| \d w_2 \int_{\mathbb{S}^2}\delta_0\big(\left( \mathbf{\Phi}_{w_2,w_1}^e(\sigma)-\chi_{v,w}^e\right)\cdot\widehat{u}\big)\d \sigma\\
&\geq \dfrac{\pi}{4\sqrt{3}R } \int_{\mathds{B}(v_1,r)} h(w_1)\d w_1 \int_{\mathds{B}(v_2,r)} g(w_2)|w_2-w_1|\d w_2.\end{split}\end{equation*}
In particular, using points \textit{(i)--(iii)} of Lemma \ref{l0} one has
\begin{equation*}
\mathcal{A}(v,w) \geq  \frac{\kappa \pi r\eta^2}{4\sqrt{3}R}  \qquad \forall v \in \mathds{B}(v_4,r), \: w \in \mathds{B}(v_3,r).
\end{equation*}
According to \eqref{eq:it}, we get for any $v \in \mathds{B}(v_4,r)$
\begin{equation*}\begin{split}
\Q^+_e(f,G)(v)
&=\dfrac{2}{\pi} \IR f(w)\Delta_e(|v-w|)\mathcal{A}(v,w)\d w \\
&\geq \frac{\kappa  r\eta^2}{2\sqrt{3}R} \int_{\mathds{B}(w_3,r)} f(w)\Delta_e(|v-w|)\d w.
\end{split}\end{equation*}
If $v \in \mathds{B}(v_4,r)$ and $w \in \mathds{B}(w_3,r)$ then $|v-w| \leq 2r+|v_4-v_3| \leq 2(r+\sqrt{3}R)$ thanks to Lemma \ref{l0} \textit{(i)}.  Using Remark \ref{Deltae}, $\Delta_e(|v-w|) \geq \frac{1}{2|v-w|}$.  Therefore,
\begin{equation*}
\Q_e^+(f,G)(v) \geq \dfrac{\kappa\,r \,\eta^2}{8(r+\sqrt{3}R)\sqrt{3}R}\int_{\mathds{B}(w_3,r)} f(w) \d w \geq \dfrac{\kappa\,r\,\eta^3}{8(r+\sqrt{3}R)^2} \qquad \forall v \in \mathds{B}(v_4,r)
\end{equation*}
which gives the result with $\eta_0=\frac{3\,r^2\,\eta^3}{2(r+\sqrt{3}R)}$.
\end{proof}
\subsection{Evolution family for the rescaled Boltzmann equation}\label{sec:lower}
According to \eqref{Thettau} there exists $C_0 >0$ such that
\begin{equation*}
\Q^-(g,g)(t,v) \leq C_0(1+|v|)g(t,v) \qquad \forall t > 0.
\end{equation*}
The rescaled equation \eqref{eqgt} can be rewritten in the following equivalent form
\begin{equation*}
\left\{
\begin{split}
\partial_t g(t,v)  +  &\xi(t) v \cdot \nabla_v g(t,v) + \Big(3\xi(t) + C_0(1+|v|)\Big)g(t,v)\\
&=\Q_{\et}^+(g,g)(t,v) + \left(C_0(1+|v|)g(t,v)-\Q^-(g,g)(t,v)\right)\\
g(0,v) &= f_0(v).
\end{split}
\right.
\end{equation*}
In particular, since $g(t,v) \geq 0$ it follows
\begin{equation}\label{gt}
\partial_t g(t,v)+ \xi(t) v \cdot \nabla_v g(t,v) + \Sigma(t,v)g(t,v) \geq \Q_{\et}^+(g,g)(t,v)
\end{equation}
where
\begin{equation*}
\Sigma(t,v)=\left(3\xi(t) + C_0(1+|v|)\right).
\end{equation*}
We introduce the characteristic curves associated to the transport operator in \eqref{gt},
\begin{equation}\label{cara}
\dfrac{\d}{\d t}X(t,s;v)=  \xi(t) \, X(t,s;v), \qquad X(s,s;v)=v,
\end{equation}
which produces a unique global solution given by
\begin{equation*}
X (t,s;v)=v\,\exp\left( \int_s^t \xi(\tau)\, \d\tau\right).
\end{equation*}
In order to simply notation let us introduce the evolution family $(\mathcal{S}_s^t)_{t \geq s  \geq 0}$ defined by
\begin{equation*}
\big[\mathcal{S}_s^t\,h\big](v):=\exp\left(-\int_s^t \Sigma\big(\tau,X(\tau;t,v)\big)\, \d\tau\right)h\big(X(s;t,v)\big) \qquad \forall t \geq s \geq 0,\;\; \forall h=h(v).
\end{equation*}
The evolution family preserves positivity, thus according to \eqref{gt} the solution $g(t,v)$ to \eqref{eqgt} satisfies the following \emph{Duhamel inequality}
\begin{equation}\label{sol}
g(t,v) \geq \left[\mathcal{S}_0^t f_0\right](v)
+ \int_0^t \big[\mathcal{S}_s^t\Q_{\es}^+ \left(g(s,\cdot),g(s,\cdot)\right)\big](v)\d s.
\end{equation}
\begin{lemme}
For any nonnegative $h=h(v) \geq 0$
\begin{equation}\label{boundSt}
\big[\mathcal{S}_s^t\,h\big](v)  \geq \left(\lambda_s^t\right)^3\exp\left(-\sigma(v)(t-s)\right)\big[\mathcal{T}_{\lambda^t_s}h\big](v)
\end{equation}
where $\sigma(v)=C_0(1+ |v|)$ and $\mathcal{T}_{\lambda_s^t}$ denotes the dilation of parameter
\begin{equation*}
\lambda_s^t=\exp\left(- \int_s^t \xi(\tau)\d \tau\right) \in (0,1),
\end{equation*}
i.e.  $\mathcal{T}_\lambda F(v)=F(\lambda v)$ for any $v \in \mathbb{R}^3$ and any $\lambda \in (0,1).$
\end{lemme}
\begin{proof} Just note that
\begin{equation*}
|X(s;t,v)| \leq |v|  \qquad \forall t \geq s \geq 0
\end{equation*}
and
\begin{equation*}
\Sigma\big(\tau,X(\tau;t,v)\big) \leq 3\xi(\tau) +  \sigma(v) \qquad \forall 0 \leq \tau \leq t, \quad \forall v \in \mathbb{R}^3.
\end{equation*}
Thus, for nonnegative $h(v)$ it follows that $\big[\mathcal{S}_s^t\,h\big](v) \geq  \left(\lambda_s^t\right)^3\exp\big(-\sigma(v)(t-s)\big)h\big(X(s,t;v)\big)$ which is the desired result.
\end{proof}
\begin{nb}\label{def:lambdast} Whenever $e(\cdot)$ is a constant coefficient of normal restitution, i.e. $\gamma=0$, one simply has $\xi(\tau)=1$ and $\lambda_s^t=\exp\big(-(t-s)\big)$.  For general  coefficients belonging to the class $\mathcal{R}_\gamma$ it follows from \eqref{intxit} that
\begin{equation}\label{lambda}
1 \geq  \lambda_{t}^{s+t} \geq \left(\lambda_0^s\right)^{\bar{c}} \qquad \forall s, t \geq 0.
\end{equation}
In particular for constant coefficient of normal restitution the inequality holds with $\bar{c}=1$.
\end{nb}
\begin{lemme}\label{SstQ}
Assume that the coefficient of normal restitution $e(\cdot)$ is of class $\mathcal{R}_\gamma$ with $\gamma >0$.  If $f=f(v) \geq 0$ is a distribution function
\begin{multline}\label{Tscal}
\T_{\lambda_s^t} \Q_{\es}^+\big(\T_{\lambda_0^s} f, \T_{\lambda_\tau^s} \Q_{\ett}^+(\T_{\lambda_0^\tau} f,\T_{\lambda_0^\tau} f)\big)=\\
(\lambda_\tau^0)^4 (\lambda_s^0)^4\,\T_{\lambda_0^t} \Q_e^+\big(\,f\,,\Q_e^+(f,f)\big)\,, \qquad \forall\;\; 0 \leq \tau \leq s \leq t.
\end{multline}
In particular, when $f$ is compactly supported with support included in $\mathds{B}(0,\varrho)$ ($\varrho >0$), then for any $t >0$ there exists $C(t,\varrho) >0$ such that
\begin{equation}\label{concl}
\Ss_s^t \Q_{\es}^+\big(\Ss_0^s f, \Ss_\tau^s \Q_{\ett}^+(\Ss_0^\tau f,\Ss_0^\tau f)\big) \geq C(t,\varrho)\T_{\lambda_0^t} \Q_e^+\left(\,f\,,\Q_e^+(f,f)\right)\,, \qquad \forall\;\; 0 \leq \tau \leq s \leq t.
\end{equation}
\end{lemme}
\begin{proof} The proof of \eqref{Tscal} is based on a repeated use of the  scaling relation
\begin{equation}\label{scalingTl}
\T_{\lambda} \Q_e^+\big(f,g\big)=\lambda^4 \Q_{\T_\lambda e}^+\big(\T_\lambda f,\T_\lambda g\big) \qquad \forall \lambda >0
\end{equation}
and the fact that
\begin{equation*}
\et(r)=\T_{\lambda_s^t}\es(r) \qquad \forall r \geq 0,\:\forall t \geq s.
\end{equation*}
Indeed, one deduces from these two identities that
\begin{equation*}
\T_{\lambda_\tau^s}\Q^+_{\ett}\big(\T_{\lambda_0^\tau} f,\T_{\lambda_0^\tau} f\big) =(\lambda_\tau^s)^4 \Q_{\es}^+\big(\T_{\lambda_0^s} f, \T_{\lambda_0^s} f\big)=(\lambda_\tau^s)^4 (\lambda_0^s)^{-4} \T_{\lambda_0^s} \Q_e^+\big(f,f\big).
\end{equation*}
Therefore,
\begin{equation*}\begin{split}
\Q_{\es}^+\big(\T_{\lambda_0^s} f, \T_{\lambda_\tau^s} \Q_{\ett}^+(\T_{\lambda_0^\tau} f,\T_{\lambda_0^\tau} f)\big)&=(\lambda_\tau^s)^4 (\lambda_0^s)^{-4}\Q_{\es}^+\big(\T_{\lambda_0^s} f,\T_{\lambda_0^s} \Q_e^+(f,f)\big)\\
&=(\lambda_\tau^s)^4 (\lambda_0^s)^{-8} \T_{\lambda_0^s} \Q_{e}^+\big(f,\Q_e^+(f,f)\big).
\end{split}
\end{equation*}
Applying $\T_{\lambda_s^t}$ to this identity we obtain \eqref{Tscal} since $\T_{\lambda_s^t}\T_{\lambda_0^s}=\T_{\lambda_0^t}$ and $(\lambda_\tau^s)^4 (\lambda_0^s)^{-8}=(\lambda_\tau^0)^4 (\lambda_s^0)^4$.\\

\noindent
Now, if $f$ is compactly supported, a repeated use of \eqref{boundSt} together with  \eqref{Tscal} yield to the result.  Specifically, if $f(v)=0$ for any $|v| \geq \varrho$, then $\Ss_s^t f(v)=0$ for any $|v| \geq \frac{\varrho}{\lambda_s^t}=\lambda_t^s \varrho$ and \eqref{boundSt} shows that
\begin{equation*}\label{est1}
\Ss_s^t f  \geq \left(\lambda_s^t\right)^3\exp\big(-\sigma(\lambda_s^t\varrho)(t-s)\big) \T_{\lambda_s^t}f.
\end{equation*}
In particular,
\begin{equation*}
\Q_{\ett}^+\big(\Ss_0^\tau f,\Ss_0^\tau f \big) \geq \left(\lambda_0^\tau\right)^6\exp\big(-2\sigma(\lambda_\tau^0\varrho)\tau\big) \Q_{\ett}^+\left(\T_{\lambda_0^\tau}f, \T_{\lambda_0^\tau}f\right).
\end{equation*}
Recall that the support of $\T_{\lambda_0^\tau}f$ is included in $\mathds{B}\big(0,\lambda_\tau^0\varrho\big)$, hence, the support of $\Q_{\ett}^+\left(\T_{\lambda_0^\tau}f, \T_{\lambda_0^\tau}f\right)$ is included in $\mathds{B}\big(0,\sqrt{2}\lambda_\tau^0 \varrho\big)$. Iterating this procedure, and computing the support at each step, we get first that
\begin{equation*}
\Ss_\tau^s \Q_{\ett}^+\big(\Ss_0^\tau f,\Ss_0^\tau f\big) \geq C_0(s,\tau,\varrho) \Q_{\es}^+\left(\T_{\lambda_0^s}f,\T_{\lambda_0^s}f\right)
\end{equation*}
with
\begin{equation*}
C_0(s,\tau,\varrho)=\exp\left(-2\sigma(\lambda_\tau^0\varrho)\tau -\sigma(\sqrt{2}\lambda_\tau^0 \varrho)(s-\tau)\right)\left(\lambda_0^\tau\right)^6\left(\lambda_\tau^s\right)^7.
\end{equation*}
Since the support of $\Q_{\es}^+\big(\Ss_0^s f, \Ss_\tau^s \Q_{\ett}^+(\Ss_0^\tau f,\Ss_0^\tau f)\big)$ is included in $\mathds{B}\big(0,2\lambda_s^0 \varrho\big)$ it follows that
\begin{equation*}
\Ss_s^t \Q_{\es}^+\big(\Ss_0^s f, \Ss_\tau^s \Q_{\ett}^+(\Ss_0^\tau f,\Ss_0^\tau f)\big)  \geq C_1(t,s,\tau,\varrho) (\lambda_0^t)^{-8} \T_{\lambda_0^t}\Q_{e}^+\big(f,\Q_e^+(f,f)\big)
\end{equation*}
with
\begin{equation*}
C_1(t,s,\tau,\varrho)=C_0(s,\tau,\varrho) \exp\big(-\sigma(\lambda_s^0 \varrho)s\big)\,\exp\big(-\sigma(2\lambda_s^0\varrho)(t-s)\big)\left(\lambda_0^s\right)^3\left(\lambda_s^t\right)^3\left(\lambda_s^t\right)^8.
\end{equation*}
In addition, since $\sigma(v)=C_0 + C_0|v|$ and $\varrho < \sqrt{2}\varrho < 2\varrho$, one gets that
\begin{equation*}
C_1(t,s,\tau,\varrho)  \geq \left(\lambda_0^t\right)^9 \lambda_\tau^t \,\lambda_s^t\exp\big(-C_0(\tau+s+t)\big)\,\exp\left(-2C_0\varrho\left(t\lambda_s^0 + s\lambda_\tau^0\right)\right).
\end{equation*}
Setting
\begin{equation}\label{CtR}
C(t,\varrho)=\left(\lambda_0^t\right)^{2} \exp\left(-3C_0 t\right) \exp\left(-4C_0\varrho t\left(1+\tfrac{\gamma}{1+\gamma}t\right)^{\frac{1}{\gamma}}\right)
\end{equation}
we finally obtain \eqref{concl}.
\end{proof}
\begin{propo}\label{propoR0}
Let $f_0$ satisfying \eqref{initial} with $f_0 \in  L^p(\R^3)$ for some $1 < p < \infty$.  Let $g(t,\cdot)$ be the solution to the rescaled equation \eqref{eqgt} with initial datum $g(0,w)=f_0(w)$. For any $\tau_1 > 0$, there exists $R_1 >0$ large enough (depending only on $f_0$) and $\mu_1 >0$ such that
\begin{equation}\label{propR0}
g(t,\cdot) \geq \mu_1 \mathbf{1}_{\mathds{B}(0,R_1)}(\cdot)\,, \qquad \forall\;\; t \geq \tau_1.
\end{equation}
Moreover, for any sequence $(\chi_k)_k \in (0,1)$ and non-decreasing sequence $(\tau_k)_k$ one has
\begin{equation}\label{gtmukRk}
g(t,\cdot) \geq \mu_k \mathbf{1}_{\mathds{B}(0,R_k)}\,, \qquad \forall\;\; t \geq \tau_k
\end{equation}
with
\begin{equation}\label{induction}
\left\{
\begin{split}
R_{k+1}&=(1-\chi_k)\ell_{\mathbf{e_{\tau_k}}}(R_k)=(1-\chi_k)R_k\sqrt{1+\tfrac{1}{4}\big(1+\mathbf{e_{\tau_k}}(R_k)\big)^2}\\
\mu_{k+1}&=\chi_k^9\; \mu_k^2 \; R_k^4 \; K_e^9\big(\lambda_0^{\tau_k} R_k\big) \; \Xi_{R_k}(\tau_{k+1}-\tau_k), \quad \forall k \in \mathbb{N}
\end{split}
\right.
\end{equation}
where we set for any $s \geq 0$ and $R > 0$,
\begin{equation*}
\Xi_R(s)=\int_0^s \left(\lambda_0^\tau\right)^3 \exp\left(-C_0(1+\sqrt{2}R)\tau\right)\d \tau.
\end{equation*}
\end{propo}
\begin{proof} We follow the approach of \cite[Theorem 4.9, Step 2]{MiMo2} introducing the appropriate
adaptation to the viscoelastic case.\\

\noindent
\textit{First Step (Proof of the initialization \eqref{propR0}).}  Let $t_0 >0$ be fixed and define $\widehat{g}_0(t,\cdot)=g(t_0+t,\cdot)$ for $t >0$, and ${G}_0=\widehat{g}_0(0,\cdot)=g(t_0,\cdot)$.  Applying Duhamel inequality \eqref{sol} twice
\begin{equation}\label{duha}
\begin{split}
\widehat{g}_0(t,\cdot) &\geq \int_0^t  \mathcal{S}_{s+t_0}^{t+t_0}\Q_{\mathbf{e_{s+t_0}}}^+ \left(\widehat{g}_0(s,\cdot),\widehat{g}_0(s,\cdot)\right) \d s\\
&\geq \int_0^t \d s \int_0^s  \mathcal{S}_{s+t_0}^{t+t_0}\Q_{\mathbf{e_{s+t_0}}}^+ \left( \mathcal{S}_{t_0}^{s+t_0} {G}_0, {\mathcal{S}_{\tau+t_0}^{s+t_0}}\Q_{\mathbf{e_{\tau+t_0}}}^+\left( {\mathcal{S}_{t_0}^{\tau+t_0}}  {G}_0,{\mathcal{S}_{t_0}^{\tau+t_0}} {G}_0,\right)\right)\d \tau,
\end{split}
\end{equation}
Using the first part of Lemma \ref{l1} to $G_0$, for $R >0$ and $\kappa >0$ large enough, there exist velocities $v_1,\, v_2$, radius $r >0$ and $\eta >0$ (which are \emph{independent} of the choice of $t_0$ according to Proposition \ref{gtprop}) such that
\begin{equation*}
\int_{\mathds{B}(v_{i},r)}G_0(v)\d v\geq\eta \quad i=1,2.
\end{equation*}
It follows from \eqref{lambda} that for any $s,\, t_0 \geq 0$ one has $1- {\lambda}_{t_0}^{s+t_0} \leq 1- (\lambda^{s}_{0})^{\bar{c}}$.  Consequently, the quantity $1- {\lambda}_{t_0}^{s+t_0}$ can be taken uniformly small independently of $t_0$.  Therefore, by a continuity argument, there exists a $T_1 >0$ small enough and \emph{independent of $t_0$} such that
\begin{equation}\label{Ftau}
\int_{\mathds{B}(v_i,r)}[\mathcal{S}_{t_0}^{s+t_0}G_0](v)\d v \geq \frac{\eta}{2} \qquad \forall\;\; i=1,2\;,\;\;\forall\;\; s \in [0,T_1].
\end{equation}
In fact, it is possible to replace everywhere in \eqref{duha} the term ${\mathcal{S}_{t_0}^{\tau+t_0}}  {G}_0$ by $F_\tau:={\mathcal{S}_{t_0}^{\tau+t_0}}  {G}_0\,\mathbf{1}_{\mathds{B}(0,R)}$.  Then, applying Lemma \ref{SstQ} we obtain
\begin{multline*}
\mathcal{S}_{s+t_0}^{t+t_0}\Q_{\mathbf{e_{s+t_0}}}^+ \left( \mathcal{S}_{t_0}^{s+t_0} {G}_0, {\mathcal{S}_{\tau+t_0}^{s+t_0}}\Q_{\mathbf{e_{\tau+t_0}}}^+\left( {\mathcal{S}_{t_0}^{\tau+t_0}}  {G}_0,{\mathcal{S}_{t_0}^{\tau+t_0}} {G}_0,\right)\right)\\
\geq C_{T_1} \T_{{\lambda}_{t_0}^{t+t_0}} \Q_e^+\left(\widehat{G}_0,\,\Q_e^+(\widehat{G}_0,\widehat{G}_0)\right)\,,  \qquad \forall 0 \leq \tau \leq s \leq t \leq T_1
\end{multline*}
with $\widehat{G}_0={G}_0\,\mathbf{1}_{\mathds{B}(0,R)}$.   From \eqref{CtR} it follows
\begin{equation*}
C_{T_1}=\inf_{t \in [0,T_1]}C(t,R) \geq \exp\big(-3C_0 T_1\big) \exp\left(-4 C_0 R\,T_1\left(1+\tfrac{\gamma}{1+\gamma}T_1\right)^{\frac{1}{\gamma}}\right).
\end{equation*}
Notice that for $T_1 >0$ small enough, one has $C_{T_1} > 1/2$.   Therefore, according to \eqref{duha}
\begin{equation*}
\widehat{g}_0(t,\cdot) \geq \frac{t^2}{4} \T_{ {\lambda}_{t_0}^{t+t_0}} \Q_e^+\left(\widehat{G}_0,\,\Q_e^+(\widehat{G}_0,\widehat{G}_0)\right)\,, \qquad \forall\;\; 0 \leq t \leq T_1.
\end{equation*}
We can apply  apply Lemma \ref{l1}  to $\widehat{G}_0$ and and conclude that there exist $r_0 >0$ and $ {\eta} >0$ such that
\begin{equation*}
\Q_e^+\left(\widehat{G}_0,\,\Q_e^+(\widehat{G}_0,\widehat{G}_0)\right) \geq { \eta}_0 \mathbf{1}_{\mathds{B}(v_4,r_0)}.
\end{equation*}
Clearly ${\lambda}_{t_0}^{t+t_0} \simeq 1$ when $t \simeq 0$. Consequently, up to reducing again $T_1$, for any $t_1 \in (0,T_1/2]$ there exists $\eta_1 >0$ and $r_1 >0$ such that
\begin{equation*}
\widehat{g}_1(t,\cdot):=\widehat{g}_0(t+t_1,\cdot) \geq \eta_1 \mathbf{1}_{\mathds{B}(v_4,r_1)}\,, \qquad \forall\;\; t \in (0,T_1/2).
\end{equation*}
Using again Duhamel's formula twice \footnote{Using Duhamel's formula twice at this stage allows to derive an estimate for $\Q_e^+$ instead of $\Q_{\et}^+$ and consequently to obtain a time uniform lower bound.}, we obtain
\begin{equation*}
\widehat{g}_1(t,\cdot) \geq \int_0^t \d s \int_0^s  \widetilde{\mathcal{S}_s^t}\Q_{\mathbf{e_{s+t_0+t_1}}}^+ \left(\widetilde{\mathcal{S}_0^s} {G}_1,\widetilde{\mathcal{S}_\tau^s}\Q_{\widetilde{e}_\tau}^+\left(\widetilde{\mathcal{S}_0^\tau}  {G}_1,\widetilde{\mathcal{S}_0^\tau} {G}_1\right)\right)\d \tau
\end{equation*}
with $\widetilde{\mathcal{S}_s^t}= \mathcal{S}_{s+t_0+t_1}^{t+t_0+t_1}$  and $G_1=\widehat{g}_1(0,\cdot)=g(t_0+t_1,\cdot) \geq  \eta_1 \mathbf{1}_{\mathds{B}(v_4,r_1)}$.  Using the procedure presented above we obtain
\begin{equation*}
\widehat{g}_1(t,\cdot)  \geq \eta_1^3\dfrac{C(T_1) t^2}{2} \T_{{\lambda}_{t_0+t_1}^{t+t_0+t_1}} \Q_e^+\left(\mathbf{1}_{\mathds{B}(v_4,r_1)} , \Q_e^+(\mathbf{1}_{\mathds{B}(v_4,r_1)}, \mathbf{1}_{\mathds{B}(v_4,r_1)})\right)\,, \qquad \forall\;\; t \in (0,T_1/2)
\end{equation*}
with constant $C(T_1) >0$ depending only on $T_1$ and $R$.  According to Proposition \ref{spread} and since $r_1 < \frac{\sqrt{5}}{2}r_1$, there exists $c_1 >0$ such that
\begin{equation*}
\Q_e^+(\mathbf{1}_{\mathds{B}(v_4,r_1)}, \mathbf{1}_{\mathds{B}(v_4,r_1)})  \geq c_1 \mathbf{1}_{\mathds{B}(v_4,r_1)}\,,
\end{equation*}
hence
\begin{equation*}
\widehat{g}_1(t,\cdot)  \geq  \eta_1^3\,c_1\dfrac{C(T_1) t^2}{2} \T_{{\lambda}_{t_0+t_1}^{t+t_0+t_1}}  \Q_e^+\left(\mathbf{1}_{\mathds{B}(v_4,r_1)} , \mathbf{1}_{\mathds{B}(v_4,r_1)}\right).
\end{equation*}
Using Proposition \ref{spread} it follows that for any $\epsilon >0$, there exists $\kappa(\epsilon) >0$ such that
\begin{equation*}
\widehat{g}_1(t,\cdot)  \geq \kappa(\epsilon)c_1 \eta_1^3\dfrac{C(T_1) t^2}{2} \T_{{\lambda}_{t_0+t_1}^{t+t_0+t_1}} \mathbf{1}_{\mathds{B}(v_4,(1+\epsilon)r_1)}\,, \qquad \forall t \in (0,T_1/2).
\end{equation*}
Arguing as above there exists $T_2 \in (0,T_1/2)$ and $t_2 \in (0,T_2/2)$ small enough, $\eta_2 >0$ and $r_2=(1+\epsilon)r_1$ such that
\begin{equation*}
\widehat{g}_2(t,\cdot):=\widehat{g}_1(t+t_2,\cdot) \geq \eta_2 \mathbf{1}_{\mathds{B}(v_4,r_2)}\,, \qquad \forall\;\; t \in (0,T_2/2).
\end{equation*}
Iterating this procedure, one concludes as in Step 3 of \cite[Theorem 4.9]{MiMo2} that there exists some explicit $\eta_\star >0$ and some arbitrarily small $t_\star >0$, both independent of the initial choice of $t_0$, such that
\begin{equation*}
g(t_\star + t_0,\cdot) \geq \eta_\star \mathbf{1}_{\mathds{B}(0,R)}.
\end{equation*}
Since $t_0 > 0$ is arbitrary, this implements the initialization step.\\

\noindent
\textit{Second step (Implementation of the induction scheme).}  From \eqref{propR0} and using Duhamel's formula \eqref{sol}
\begin{equation*}
g(t,\cdot) \geq \int_0^t \Ss_s^t \Q_{\es}^+\big(g(s,\cdot),g(s,\cdot)\big)\d s
\end{equation*}
and whenever $t \geq \tau_1$
\begin{equation}\label{gtDuh}
g(t,\cdot) \geq \int_{\tau_1}^t \Ss_s^t \Q_{\es}^+\big(g(s,\cdot),g(s,\cdot)\big)\d s \geq \mu_1^2 \int_{\tau_1}^t \Ss_s^t \Q_{\es}^+ \left( \mathbf{1}_{\mathds{B}(0,R_1)}\,,\, \mathbf{1}_{\mathds{B}(0,R_1)}\right)\d s.
\end{equation}
Since the support of $\Q_{\es}^+ \left( \mathbf{1}_{\mathds{B}(0,R_1)}\,,\, \mathbf{1}_{\mathds{B}(0,R_1)}\right)$ is included in $\mathds{B}\big(0,\ell_{\es}(R_1)\big)$ we get from \eqref{boundSt} that
\begin{equation}
\begin{split}\label{SstQe+}
\Ss_s^t \Q_{\es}^+ \left( \mathbf{1}_{\mathds{B}(0,R_1)}\,,\, \mathbf{1}_{\mathds{B}(0,R_1)}\right) &\geq \left(\lambda_s^t\right)^3
 \exp\big(-\sigma(\ell_{\es}(R_1))(t-s)\big)\times\\
 &\phantom{++++ +++}\T_{\lambda_s^t}\Q_{\es}^+\left( \mathbf{1}_{\mathds{B}(0,R_1)}\,,\, \mathbf{1}_{\mathds{B}(0,R_1)}\right).
 \end{split}
 \end{equation}
The scaling properties of $\Q_e^+$ gives for any $s \in (\tau_1,t)$
\begin{equation*}
\T_{\lambda_s^t}\Q_{\es}^+\left( \mathbf{1}_{\mathds{B}(0,R_1)}\,,\, \mathbf{1}_{\mathds{B}(0,R_1)}\right)=\left(\lambda_s^t\right)^4 \Q_{\et}^+ \left(\mathbf{1}_{\mathds{B}(0,\lambda_t^s R_1)},\mathbf{1}_{\mathds{B}(0,\lambda_t^s R_1)}\right),
\end{equation*}
and thanks to Proposition \ref{spread} and with the notation of Remark \ref{nbImpo}, for any $\chi_1 \in (0,1)$
\begin{equation}\label{qet+}
\Q_{\et}^+ \left(\mathbf{1}_{\mathds{B}(0,\lambda_t^s R_1)},\mathbf{1}_{\mathds{B}(0,\lambda_t^s R_1)}\right) \geq \left(\lambda_t^s\right)^4\; R_1^4\; \chi_1^9\; \Psi_{\et}^9(\lambda_t^s\,R_1) \; \mathbf{1}_{\mathds{B}(0,(1-\chi_1)\ell_{\et}(\lambda_t^s R_1))}.
\end{equation}
Using the fact that $\et(\cdot)=\T_{\lambda_0^t}e(\cdot)$  one can check that
\begin{equation*}
\Psi_{\et}(\lambda_t^s\,R_1)=K_e(\lambda_0^s \,R_1),
\end{equation*}
in other words,
\begin{equation*}
\Q_{\et}^+ \left(\mathbf{1}_{\mathds{B}(0,\lambda_t^s R_1)},\mathbf{1}_{\mathds{B}(0,\lambda_t^s R_1)}\right) \geq \chi_1^9\;(\lambda_t^s)^4\; R_1^4\; K_e^9(\lambda_0^{s}\,R_1)\; \mathbf{1}_{\mathds{B}(0,(1-\chi_1)\ell_{\et}(\lambda_t^s R_1))}\,, \qquad \forall t \geq 0.
\end{equation*}
Then, since $\ell_{\et}(\lambda_t^s R_1) =\lambda_t^s \ell_{\es}(R_1)$, we obtain
\begin{equation*}
\T_{\lambda_s^t}\Q_{\es}^+\left( \mathbf{1}_{\mathds{B}(0,R_1)}\,,\,  \,\mathbf{1}_{\mathds{B}(0,R_1)}\right) \geq \chi_1^9\; R_1^4\; K_e^9(\lambda_0^{s}\,R_1) \; \mathbf{1}_{\mathds{B}(0,(1-\chi_1)\lambda_t^s \ell_{\es}(R_1))}.
\end{equation*}
Notice also that $\lambda_t^s \geq 1$, thus
\begin{equation*}
\mathbf{1}_{\mathds{B}(0,(1-\chi_1)\lambda_t^s \ell_{\es}(R_1))} \geq \mathbf{1}_{\mathds{B}(0,(1-\chi_1)\ell_{\es}(R_1))},
\end{equation*}
and using \eqref{SstQe+}
\begin{equation*}
\Ss_s^t \Q_{\es}^+ \left( \mathbf{1}_{\mathds{B}(0,R_1)}\,,\, \mathbf{1}_{\mathds{B}(0,R_1)}\right) \geq \left(\lambda_s^t\right)^3 \chi_1^9\; R_1^4\; K_e^9(\lambda_0^{s}\,R_1)\; \exp\big(-\sigma(\ell_{\es}(R_1))(t-s)\big)\;\mathbf{1}_{\mathds{B}(0,(1-\chi_1)\ell_{\es}(R_1))}.
\end{equation*}
According to \eqref{gtDuh} we obtain
\begin{equation*}
g(t,\cdot) \geq \mu_1^2\;  \chi_1^9\; R_1^4\;   \int_{\tau_1}^t \exp\big(-\sigma(\ell_{\es}(R_1))(t-s)\big)\,K_e^9(\lambda_0^{s}\,R_1) \left(\lambda_s^t\right)^3 \,\mathbf{1}_{\mathds{B}(0,(1-\chi_1)\ell_{\es}(R_1))} \d s.
\end{equation*}
Since $s \mapsto \lambda_0^s$ and $K_e(\cdot)$ are both non-increasing, one has $K_e(\lambda_0^s\,R_1) \geq K_e(\lambda_0^{\tau_1}\,R_1)$ for any $s \in (\tau_1,t)$. Moreover, to avoid the integration in the last indicator function we simply notice that since $e(\cdot)$ is non-increasing, the mapping
\begin{equation*}
s \mapsto \ell_{\es}(R_1) \quad \text{ is non-decreasing for any } R_1 >0.
\end{equation*}
Therefore, $\ell_{\es}(R_1) \geq \ell_{\mathbf{\mathbf{e_{\tau_1}}}}(R_1)$ for any $s \in (\tau_1,t)$ and
\begin{equation*}
g(t,\cdot) \geq \mu_1^2\;  \chi_1^9\; R_1^4\; K_e^9(\lambda_0^{\tau_1}\,R_1) \; \mathbf{1}_{\mathds{B}(0,R_2)} \int_{\tau_1}^t \left(\lambda_s^t\right)^3 \exp\left(-\sigma(\ell_{\es}(R_1))(t-s)\right)\d s
\end{equation*}
with $R_2=(1-\chi_1)\ell_{\mathbf{e_{\tau_1}}}(R_1)$.  We use the obvious estimate $\ell_{\es}(R_1) \leq C_0(1+\sqrt{2}R_1)$ and the fact that $\lambda_s^t \geq \lambda_0^{t-s}$ to obtain
\begin{equation*}
g(t,\cdot) \geq  \mu_1^2\;  \chi_1^9\; R_1^4\; K_e^9(\lambda_0^{\tau_1}\,R_1) \; \mathbf{1}_{\mathds{B}(0,R_2)}  \int_0^{t-\tau_1} \left(\lambda_0^\tau\right)^3 \exp\left(-C_0(1+\sqrt{2}R_1)\tau\right)\d \tau.
\end{equation*}
Therefore,
\begin{equation*}
g(t,\cdot) \geq \mu_2 \,\mathbf{1}_{\mathds{B}(0,R_2) }\,, \qquad \forall\;\; t \geq \tau_2 > \tau_1
\end{equation*}
with $R_2=(1-\chi_1)\ell_{\mathbf{e_{\tau_1}}}(R_1)$ and
\begin{equation*}
\mu_2 =\mu_1^2\;  \chi_1^9\; R_1^4\; K_e^9(\lambda_0^{\tau_1}\,R_1)\; \Xi_{R_1}(\tau_2-\tau_1).
\end{equation*}
Repeating the argument we obtain the result.
\end{proof}
\subsection{Conclusion of the proof} We are now in position to conclude with the uniform exponential lower pointwise bounds.
\begin{proof}[Proof of Theorem \ref{lowVisco0}]
We apply Proposition \ref{propoR0} to a constant sequence $(\chi_k)_k$ and bounded sequence $(\tau_k)_{k \geq 1}$.   More precisely, let $t_1 > 0$ be fixed and write
\begin{equation*}
\tau_1=\frac{t_1}{2}, \qquad \tau_{k+1}=\tau_{k}+\frac{t_1}{2^{k+1}} \qquad \forall k \geq 1.
\end{equation*}
For any given $\varepsilon >0$, set $\chi_k=\varepsilon$ for all $k \geq 1$ and let
\begin{equation*}
b_\varepsilon=(1-\varepsilon)\ell_{e_0}(1)=(1-\varepsilon)\sqrt{1+\left(\tfrac{1+e_0}{2}\right)^2}.
\end{equation*}
Since $e_0=\inf_{r > 0}e(r)$ (recall that $e(\cdot)$ is non-increasing), one deduces from \eqref{induction} that
\begin{equation*}
(\sqrt{2})^{k-1} R_1 \geq R_k \geq (b_\varepsilon)^{k-1} R_1\,, \qquad \forall\;\; k \geq 1.
\end{equation*}
Moreover, by definition of $\Xi_R(s)$
\begin{equation*}
\Xi_R(s) \geq s\left(\lambda_0^s\right)^3 \exp\big(-C_0(1+\sqrt{2}R)s\big)\,, \qquad \forall\;\; s > 0, \;R > 0.
\end{equation*}
Therefore, there exist two constants $\alpha >0$ and $c  >0$ (both depending on $t_1$) such that
\begin{equation*}
\Xi_R(s) \geq c\, s \exp(-\alpha R)\,, \qquad \forall\;\; 0 \leq  s \leq t_1.
\end{equation*}
In particular, since $\tau_{k+1}-\tau_k=\frac{t_1}{2^{k+1}} \leq t_1$, one gets that
\begin{equation*}
\Xi_{R_k}(\tau_{k+1}-\tau_k) \geq c \,\left(\tau_{k+1}-\tau_k\right)\,\exp\left(-\alpha R_k\right)\,,   \qquad \forall\;\; k \geq 1.
\end{equation*}
Using \eqref{induction} it follows that,
\begin{equation}\label{estiMuk}
\mu_{k+1} \geq \dfrac{\varepsilon^9\,c  t_1}{2^{k+1}}R_k^4 \exp\left(-\alpha R_k\right)\big(K_e(\lambda_0^{\tau_k} R_k)\big)^9 \mu_k^2.
\end{equation}
We distinguish two cases:\\

\noindent
\textit{Case $e_0 > 0$.}  According to Remark \ref{rem:Cases},  there exists $C > 0$ such that
\begin{equation*}
K_e(x) \geq C \qquad \forall x > 0
\end{equation*}
and since $1 \leq R_k \leq R_1 2^{\tfrac{k-1}{2}}$  for any $k\geq1$, there exists $\kappa_\varepsilon > 0$ (depending only on $\varepsilon$) such that
\begin{equation*}
\mu_{k+1} \geq \dfrac{\kappa_\varepsilon t_1}{2^{k+1}} \exp\left(-\alpha  R_1 2^{\frac{k-1}{2}}\right)\,\mu_k^2\,, \qquad \forall\;\; k \geq 1.
\end{equation*}
Therefore,
\begin{equation*}
\mu_{k+1} \geq \left(\kappa_\varepsilon\,  t_1\right)^{
\sum_{j=0}^{k-1}2^j} \exp\left(-\alpha  R_1 \sum_{j=0}^{k-1}2^{j+\frac{k-j-1}{2}}\right)
2^{-\sum_{j=0}^{k-1} 2^j\left(k+1-j\right)} \mu_1^{2^k}
\end{equation*}
from which we deduce that
\begin{equation*}
\mu_k \geq A^{2^k}\,, \qquad \forall\;\; k \geq 1 \qquad \text{with} \quad A:=\dfrac{\sqrt{\kappa_\varepsilon\, t_1\,\mu_1}}{2}\exp(-\alpha_0\,R_1) <1
\end{equation*}
for $\varepsilon$ small enough.  Using this estimate in  \eqref{gtmukRk} we obtain
\begin{equation*}
g(t,\cdot) \geq A^{2^k} \mathbf{1}_{\mathds{B}(0,R_k)} \geq A^{2^k} \mathbf{1}_{\mathds{B}(0,(b_\varepsilon)^{k-1} R_1)}\,,\quad \forall\;\; t \geq \tau_k.
\end{equation*}
Since $\tau_k=(1-2^{-k})t_1 < t_1$, we obtain that for any $t \geq t_1$
\begin{equation*}
g(t,\cdot) \geq A^{2^k} \mathbf{1}_{\mathds{B}(0,(b_\varepsilon)^{k-1} R_1)}\,, \qquad \forall\;\; k \geq 1.
\end{equation*}
For $\varepsilon > 0$ small enough one has $b_\varepsilon^{a_0} > 2$.  Also, for any $v \in \R^3$ there exists $k \geq 1$ such that  $2^{k-2} R_1^{a_0} \leq |v|^{a_0} < 2^{k-1} R_1^{a_0}$.  Therefore,
setting $c_0:=-\frac{4\log A}{R_1^{a_0}} > 0$ we conclude that
\begin{equation*}
g(t,v) \geq \exp(2^k \log A) \geq \exp(-c_0\,|v|^{a_0})\,, \qquad \forall\;\; t \geq t_1.
\end{equation*}
Putting this together with \eqref{propR0} yields the desired estimate.\\

\noindent
\textit{Case $e_0=0$.}  According to Assumption \eqref{assKe} and Remark \ref{rem:Cases}, one knows that there exists $n \geq 0$ and $C >0$ such that
\begin{equation*}
K_e(x) \geq C x^{1-n}
\end{equation*}
and \eqref{estiMuk} yields
\begin{equation*}
\mu_{k+1} \geq \dfrac{\kappa_\varepsilon  t_1}{2^{k+1}}R_k^{13-9n} \exp\left(-\alpha R_k\right) \mu_k^2\,, \qquad \forall\;\; k \geq 1
\end{equation*}
for some $\kappa_\varepsilon=(C \varepsilon)^9 c_0$ depending only on $\varepsilon$. Since $R_k \geq 1$, it is possible to find $\alpha_0 > \alpha$ such that $R_k^{13-9n}\exp\big(-\alpha R_k\big) \geq \exp\big(-\alpha_0 R_k\big)$ for any $k \geq 1$.   Additionally, $R_k \leq R_1 2^{\tfrac{k-1}{2}}$ which leads to
\begin{equation*}
\mu_{k+1} \geq \dfrac{\kappa_\varepsilon t_1}{2^{k+1}} \exp\left(-\alpha_0 R_1 2^{\frac{k-1}{2}}\right)K_e^9(\lambda_0^{\tau_k} R_k) \mu_k^2.
\end{equation*}
The conclusion follows as in the previous case by noticing that in such instance
\begin{equation*}
b_\varepsilon=\tfrac{\sqrt{5}}{2}(1-\varepsilon).
\end{equation*}
\end{proof}
\begin{nb} For a constant coefficient of normal restitution $e \equiv e_0$, Theorem \ref{lowVisco0} improves the lower bound obtained in \cite{MiMo3}.  In particular, it shows that in quasi-elastic regime $e_0 \simeq 1$ the lower bound becomes ``almost Maxwellian'' in the following sense: for any $\delta >0$, there is an explicit value $\alpha=\alpha(\delta)$ such that $e_0 \in (\alpha,1) \implies a_0 \in (2,2+\delta)$.
\end{nb}
\begin{proof}[Proof of Theorem \ref{theo:Main}] The proof consists only in showing that Assumptions \ref{hyp:condi} are met and then apply Theorem \ref{main}. According to Theorem \ref{lowVisco0}, under the hypothesis of Theorem \ref{theo:Main}, the solution $g(t,\cdot)$ satisfies \eqref{hyp:lower}.  From Theorem \ref{theo:regu}, fixing $m_0\in\mathbb{N}$ and $k>0$ such that
$f_0 \in \mathbb{H}^{m_0}_k$ then, the solution $g(t,\cdot)$ to \eqref{eqgt} is such that $\sup_{t \geq 0}\|g(t)\|_{\mathbf{H}^{m_0}_{k'}} < \infty$ for some $k'\leq k$.  Choosing the regularity $m_0$ and the moments $k$ large enough so that Theorem \ref{theo:Vill} applies, one gets the conclusion from Theorem \ref{main}.
\end{proof}

\section{Comments and Perspectives}\label{sec:discuss}
We elucidate a bit more in this section on a couple of interesting issues complementary to the work done in previous Sections.  First we address the issue of initial data with lower order regularity and algebraic rate of convergence and second, the issue of exponential rate of convergence under the special regime of weak inelasticity. These two problems are delicate and we only sketch some possible paths to improve the results.

\subsection{Initial data with lower order regularity}
In our main result, Theorem \ref{theo:Main}, the initial datum $f_0$ is assumed to be very regular, more specifically, lying in some Sobolev space $\mathbf{H}^{m_0}$ for $m_0 > 0$ explicit, finite but possibly large.  We discuss here the extension of our work to rougher initial datum; the underlying conclusion being that the rougher $f_0$ is, the weaker the rate of convergence becomes (i.e. $a$ decreases in \eqref{convergence}).\\

\noindent
We consider a strong rescaled solution $g(t,v)$ of \eqref{eqgt} satisfying
\begin{equation}\label{dimcon}
c_0\leq \|g(t)\|_{\infty}\leq C_0\quad \text{and } \quad \quad \|g(t)\|_{\mathbb{H}^{1}_{\mathfrak{q}+\gamma+4}}\leq C_0,\quad \forall\;t \geq t_0>0,
\end{equation}
for some positive and finite constants $c_0$ and $C_0$.  Here $\mathfrak{q} \geq 2$ is the constant in \eqref{hyp:lower} and $\gamma >0$ has the usual connotation while $t_0$ is any positive time.  Conditions \eqref{dimcon} are used only to control $\mathcal{I}_{1}(t)$ and $\mathcal{I}_{2}(t)$ in Lemma \ref{lem:I1}.  Furthermore, under these conditions on $g$ the inequality
\begin{equation*}
\D_{1}(t)\geq C_{t_0}\,\mathcal{H}(t)^{1+\varepsilon_0}\,,\quad \forall\;t \geq t_0
\end{equation*}
holds for some (possibly large) $\varepsilon_0$ and constant $C_{t_0}$ depending on $t_0$ and weighted $L^{2}$-norms of $g$ (see \cite[page 697]{Toscani}).  Then, Proposition \ref{converEntropy} changes simply to
\begin{equation*}
\mathcal{H}(t)\leq C\,\xi(t)^{ \tfrac{ 1 }{ 1+\varepsilon_0 } }\,,\quad \forall\;t\geq t_0.
\end{equation*}
Hence, Theorem \ref{main} is valid with the reduced algebraic rate $a=\tfrac{1}{4(1+\varepsilon_0)}$.  Given our study of propagation of regularity, condition \eqref{dimcon} is satisfied assuming
\begin{equation*}
0<f_0\in L^{\infty}\cap\mathbb{H}^{1}_{\mathfrak{q}'+\gamma+4}.
\end{equation*}
for some $\mathfrak{q}'\geq\mathfrak{q}$.  Thus, we traded off regularity of the initial datum for rate of convergence towards the universal steady state.\\

\noindent
Notice that it is possible to further weaken the regularity assumption on $f_0$ by employing a technique used in the literature based on the decomposition of the solution in smooth and remainder parts $g=g^{S}+g^{R}$ with the remainder $g^{R}$ vanishing at exponential rate, see for instance \cite[Chapters 5 and 6]{MoVi} for details and references in the case of homogeneous elastic Boltzmann.  The idea is simple and consists in writing any weak solution with the evolution family similar as we did in section \ref{sec:lower}
\begin{equation*}
g(t,\cdot) =\int_0^t \big[\mathcal{V}_s^t\Q_{\es}^+ \big(g(s,\cdot),g(s,\cdot)\big)\big] \d s+\left[\mathcal{V}_0^t f_0\right]
=:g^{S}(t,\cdot)+g^{R}(t,\cdot).
\end{equation*}
This provides the decomposition (here  $(\mathcal{V}_s^t)_{t\geq s\geq 0}$ is  some suitable evolution family slightly different from $(\mathcal{S}_s^t)_{t\geq s\geq 0}$). It is not difficult to prove that $g^{R}$ indeed vanishes exponentially fast in time and that $g^{S}$ has a uniform exponential lower bound provided $g$ has one (for any $t\geq t_0$).  The fact that $g^{S}$ enjoys $\mathbb{H}^{1}$ regularity is more cumbersome to carry out and it is related with the smoothing properties of the gain Boltzmann collision operator when $g\in L^{p}$ for some $p>1$ .  There are different results on this respect for the elastic gain Boltzmann operator, see for instance \cite[Theorem 4.2]{Wennberg} or \cite[Theorem 2.1]{BandD}.  However, there is no such result for a collision operator associated to variable coefficient of normal restitution; one of the main obstacles being the use of Fourier transform in \cite{BandD} which can be readily extended to constant coefficient $e$ but becomes technically involved for non constant $e(\cdot)$.  We nevertheless expect similar result to hold under Assumption \ref{HYPdiff}.  In this way, using the decomposition technique and a natural modification of the method of Section \ref{sec:condi} (which is suited to weak solutions) the algebraic rate of convergence can, hopefully, be proven with initial datum having compact support or strong decay at infinity  and satisfying $f_0 \in L^{2}$.
\subsection{Weakly inelastic regime}
Let us ponder a bit about the possibility of improving the result given in Theorem \ref{theo:Main} in terms of the convergence rate.  First, we emphasize that in general the convergence rate of $g(t)$ towards $\M_0$ is likely to be not better than algebraic.  The reason is that such rate of convergence is tied up to the convergence rate of the viscoelastic operator $\Q_{\et}$ towards the elastic operator $\Q_1$ which  under the analysis given in \cite{AloLo3} is likely to be at most algebraic.  However, similar to the case for constant coefficient of normal restitution \cite{MiMo3}, it is possible that under a weakly inelastic regime such rate can be upgraded to exponential.  Following \cite{AloLo2}, a viscoelastic granular gas lies in this regime whenever
\begin{equation*}
\ell_\gamma(e) \ll 1 \qquad \text{ where }  \quad \ell_\gamma(e)=\sup_{r\geq0} \dfrac{1-e(r)}{r^\gamma}.
\end{equation*}
Notice that $\ell_\gamma(e)$ is finite for any $e(\cdot)$ belonging to the class $\mathcal{R}_\gamma$.  In particular note that for the true viscoelastic model, for which $e(\cdot)$ is given by \eqref{visco}, the weakly inelastic regime corresponds to $\mathfrak{a} \ll 1$.  Additionally, we have the identity
\begin{equation*}
\ell_\gamma(\et)=\xi(t)\ell_\gamma(e)\,, \qquad \forall\; t > 0,
\end{equation*}
which implies that if the viscoelastic granular gas lies in the weakly inelastic regime then $\sup_{t \geq 0}\ell_{\gamma}(\et) \ll 1$ and the operator  $\Q_{\et}$ should act as \emph{a uniform in time perturbation} of $\Q_1$.  In other words, in such regime the viscoelastic particles interact mainly elastically.   Anyhow, a linear perturbation analysis similar to the one performed in \cite{MiMo3} should be carried out for the viscoelastic model and the spectral properties of the elastic operator may allow to recover an exponential convergence towards $\M_0$ in the viscoelastic case.
\appendix

\section{Functional toolbox on the Boltzmann collision operator}
\def \la {\lambda}
\def \el {e_\la}
We recall several of the results concerning the collision operator $\Q_e$ for variable restitution coefficient quoted mostly from from \cite{AloLo3}.  Recall that $\Q_{e_\lambda}$ is the collision operator associated to the restitution coefficient $e_\lambda(r)=e(\lambda r)$ for any $r \geq 0$.
\begin{propo}\label{propo:sobdiff}
For any $\ell \in \mathbb{N}$ and $k\geq0$ there exists $C(\gamma,k,\ell)$ such that
\begin{multline*}
\|\Q^+_{e_\lambda}(f,g)-\Q^+_1(f,g)\|_{\mathbb{H}^{\ell}_{k}}
 \leq C(\gamma,k,\ell)\; \lambda^\gamma\; \left(\|f\|_{\mathbb{W}^{\ell,1}_{k+\gamma+2}}\,
 \|g\|_{\mathbb{H}^{\ell+1}_{k+\gamma+2}}+\|f\|_{\mathbb{H}^{\ell+1}_{k+\gamma+2}}\,\|g\|_{\mathbb{W}^{\ell,1}_{k+\gamma+2}}\right)
\end{multline*}
holds for any $\lambda \in [0,1].$
\end{propo}
\noindent
This estimate can be extended to weighted $L^1(m_a)$ spaces with exponential weight
\begin{equation}\label{malpha}
m_{a}(v):=\exp\left(a |v|\right), \quad v \in \R^3,\;\; a \geq 0.
\end{equation}
One quote from \cite[Corollary 3.12]{AloLo3} the following
\begin{propo}\label{preciserate} There exists \emph{explicit} $\la_0 \in (0,1)$ such that for any $a\geq0$ there exists an explicit constant $C(\gamma,a) >0$ for which the following holds
\begin{equation}\label{L1diffK}
\| \Q^+_{\el}(f,g)-\Q^+_1(f,g) \|_{L^1( m_{a})}
\leq C(\gamma,a) \la^\gamma\|f\|_{L^{1}_{k}(m_a)}\,\|g\|_{\mathbb{W}^{1,1}_k(m_a)}\,, \qquad \forall\;\; \la \in (0,\la_0)
\end{equation}
and
\begin{equation*}
\| \Q^+_{\el}(f,g)-\Q^+_1(f,g) \|_{L^1( m_{a})}
\leq C(\gamma,a) \la^{\gamma}\|g\|_{L^{1}_{k}(m_a)}\,\|f\|_{\mathbb{W}^{1,1}_k(m_a)}\,, \qquad \forall\;\; \la \in (0,\la_0)
\end{equation*}
where $k=\gamma+\frac{10}{3}.$
\end{propo}
\noindent
Notice that the above results do not cover estimates of $\Q_{\el}(f,g)-\Q_1(f,g)$ in weighted $L^1$ spaces with \emph{algebraic weight}. However, such estimates are easily deduced (maybe in a non-optimal way) from Proposition \ref{propo:sobdiff}  and the general estimate valid for any $h=h(v)$
\begin{equation}
\label{taug21}\|h\|_{L^1_k} \leq M_\theta\|h\|_{L^2_{k+3/2+\theta}} \quad \text{with} \quad M_\theta=\|\langle \cdot \rangle^{-3/2-\theta}\|_{L^2}\,, \quad \forall\; k \geq 0 \quad \forall\; \theta >0.
\end{equation}
Namely, one easily has
  \begin{propo}\label{diffL1k}
For any $k\geq0$ and any $\theta >0$, there exists $C(\gamma,k,\theta)$ such that
\begin{multline*}
\|\Q^+_{e_\lambda}(f,g)-\Q^+_1(f,g)\|_{L^1_{k}}\\
 \leq C(\gamma,k,\theta)\; \lambda^\gamma\; \left(\|f\|_{L^{1}_{k+\gamma+\theta+7/2}}\,
 \|g\|_{\mathbb{H}^{1}_{k+\gamma+\theta+7/2}}+\|f\|_{\mathbb{H}^{1}_{k+\gamma+\theta+7/2}}\,\|g\|_{L^{1}_{k+\gamma+\theta+7/2}}\right)
\end{multline*}
holds for any $\lambda \in [0,1].$
\end{propo}
\noindent
We end this section with the regularity properties of $\Q^+_e(f,g)$, see \cite[Theorem 2.5]{AloLo3}.
\begin{theo}\label{regularite}
Assume that $r \mapsto e(r)$ is of class $\mathcal{C}^m(0,\infty)$ with $\sup_{r \geq 0} re^{(k)}(r) < \infty$ for any $k=1,\ldots,m$ for some $m \geq 2$. Then, for any $\varepsilon >0$ and $\eta \geq 0$, there exists $C_e=C(e,\varepsilon,\eta)$ such that
\begin{multline}\label{regestim}
\|\Q_e^+(f,g)\|_{\mathbb{H}^{s+1}_\eta} \leq C_e\,\|g\|_{\mathbb{H}^s_{2\eta+s+3}}\|f\|_{L^1_{2\eta+s+3}} + \varepsilon\| f\|_{\mathbb{H}^s_{\eta+3}}\,\|g\|_{\mathbb{H}^s_{\eta+1}}\\
+\varepsilon\left( \| g\|_{L^1_{\eta+1}}\,\|\partial^\ell f\|_{L^2_{\eta+1}}+\| f\|_{L^1_{\eta+1}}\,\| \partial^\ell g\|_{L^{2}_{\eta+1}}\right)\,, \qquad \forall |\l|=s+1 \leq m-1.
\end{multline}
\end{theo}
\begin{nb} Note that all these results apply to the time-dependent operator $\Q_{\et}^+(f,g)$, $t \geq 0$.  Indeed, recall that
\begin{equation*}
\et(r)=e(z(t)\,r), \qquad z(t) \propto \left(1+\tfrac{\gamma}{1+\gamma}t\right)^{-1/\gamma}
\end{equation*}
where we adopted the notations of Section \ref{sec:preli}. In particular,  using the observation of \cite[Remark 2.6]{AloLo1} it follows that $\sup_{\la \in (0,1) }C(\el,\varepsilon,\eta) < \infty$ for any $\varepsilon > 0$ and $\eta > 0$ where $C(e,\varepsilon,\eta)$ is the constant appearing in the Theorem \ref{regularite}.  This allows to obtain uniform it time bounds for $\|\Q_e^+(f,g)\|_{\mathbb{H}^{s+1}_\eta}$.
\end{nb}


\begin{thebibliography}{99}

\bibitem{AlonsoIumj}
\textsc{Alonso, R. J.,} Existence of global solutions to the Cauchy problem for the inelastic Boltzmann equation with near-vacuum data, \textit{Indiana Univ. Math. J.}, \textbf{58} (2009), 999--1022.

\bibitem{AloCar}
\textsc{Alonso, R. J., Carneiro, E. \& Gamba, I. M.,} Convolution inequalities for the Boltzmann
collision operator, \emph{Comm. Math. Phys.}, \textbf{298} (2010), 293--322.

\bibitem{AloLo1}
\textsc{Alonso, R. J. \& Lods, B.} Free cooling and high-energy tails of granular gases with variable restitution coefficient, \textit{SIAM J. Math. Anal.} \textbf{42} (2010) 2499--2538.

\bibitem{AloLo2}
\textsc{Alonso, R. J. \& Lods, B.} Two proofs of Haff's law for dissipative gases: the use of entropy and the weakly inelastic regime, \textit{J. Math. Anal. Appl.} \textbf{397} (2013) 260--275.

\bibitem{AloLo3}
\textsc{Alonso, R. J. \& Lods, B.} Uniqueness and regularity of steady states of the Boltzmann equation for viscoelastic hard-spheres driven by a thermal bath, \textit{Commun. Math. Sci.}, {\bf 11} (2013), 851--906.

\bibitem{CarGamba}
\textsc{Bobylev, A. V., Carrillo, J. A. \& Gamba, I. M.}, {On some properties of kinetic and hydrodynamic equations for inelastic interactions}, \textit{J. Statist. Phys.},  \textbf{98} (2000), {743--773}.

\bibitem{BandD} \textsc{Bouchut, F. \& Desvillettes, L.}, A proof of the smoothing properties of the positive part of Boltzmann's kernel, \textit{Rev. Mat. Iberoamericana}, {\bf 14} (1998), 47--61.

\bibitem{BrPo}
\textsc{Brilliantov,  N. V. \&   P\"{o}schel,  T.}, \textbf{Kinetic
theory of granular gases}, Oxford University Press, 2004.

\bibitem{Carle}
\textsc{Carleman, T.}, \textbf{Probl\`emes math\'ematiques dans la th\'eorie cin\'etique des
              gaz}, {Publ. Sci. Inst. Mittag-Leffler. 2}, Uppsala, 1957.

\bibitem{GaPaVi} \textsc{Gamba, I., Panferov, V. \&  Villani, C.}, {On the
Boltzmann equation for diffusively excited granular media},
\textit{Comm. Math. Phys.} {\bf 246} (2004), 503--541.

\bibitem{GaPaVi2} \textsc{Gamba, I., Panferov, V. \&  Villani, C.}, {Upper Maxwellian Bounds for the Spatially Homogeneous Boltzmann Equation},
\textit{Arch. Rational Mech. Anal.} \textbf{194} (2009) 253--282.

\bibitem{MMR}\textsc{Mischler, S., Mouhot, C. \&  Rodriguez Ricard, M.}, Cooling process for inelastic
{B}oltzmann equations for hard-spheres, Part~I: The Cauchy
problem, \textit{J. Statist. Phys.} {\bf 124} (2006), 655-702.

\bibitem{MiMo}\textsc{Mischler, S. \& Mouhot, C.}, Cooling process for inelastic {B}oltzmann equations for hard-spheres, Part~II: Self-similar solution and tail behavior,
\textit{J. Statist. Phys.} {\bf 124} (2006), 655-702.

\bibitem{MiMo2}\textsc{Mischler, S. \& Mouhot, C.}, Stability, convergence to self-similarity and elastic limit for the Boltzmann equation for inelastic hard-spheres.  \textit{Comm. Math. Phys.}  \textbf{288}  (2009),  431--502.

\bibitem{MiMo3}\textsc{Mischler, S. \& Mouhot, C.}, Stability, convergence to the steady state and elastic limit for the Boltzmann equation for diffusively excited granular media.  \textit{Discrete Contin. Dyn. Syst.  A}  \textbf{24}  (2009),   159--185.

\bibitem{Mo} \textsc{Mouhot, C.,}  {Rate of convergence to equilibrium for the spatially
homogeneous Boltzmann equation with hard potentials} \emph{Comm. Math. Phys.,} {\bf 261}  (2006), 629--672.

\bibitem{MoVi} \textsc{Mouhot, C. \& Villani, C.},  Regularity theory for the
spatially homogeneous {B}oltzmann equation with cut-off,
\textit{Arch. Ration. Mech. Anal.} {\bf 173} (2004), 169--212.

\bibitem{ada}
\textsc{Pulvirenti, A. \& Wennberg, B.},  {A Maxwellian lower bound for solutions to the Boltzmann equation}, \emph{Commun. Math. Phys.,} {\bf 183}  (1997), 145--160.

\bibitem{PoSc}
{\sc Schwager, T. \& P\"{o}schel, T.}, Coefficient of normal restitution of viscous particles and cooling rate of granular gases, \textit{Phys. Rev. E} {\bf 57} (1998), 650--654.

\bibitem{Toscani}
\textsc{Toscani, G. \& Villani, C.}, Sharp Entropy Dissipation Bounds and Explicit Rate of Trend to Equilibrium for the Spatially Homogeneous Boltzmann Equation, \emph{Commun. Math. Phys.}, {\bf 203} (1999), 667--706.

\bibitem{VillCanc}
\textsc{Villani, C.}, Regularity estimates via the entropy dissipation for the spatially homogeneous Boltzmann equation without cut-off, \emph{Rev. Mat. Iberoamericana} {\bf 15} (1999), 335--352.

\bibitem{villcerc}
\textsc{Villani, C.,}  {Cercignani's conjecture is sometimes true and always almost true},  \emph{Comm. Math.
Phys}, {\bf 234}  (2003),  455--490.

\bibitem{Wennberg}
\textsc{Wennberg, B.}, Regularity in the Boltzmann equation and the Radon transform, \textit{Comm. Partial Differential Equations} \textbf{19} (1994), 2057--2074.

\end{thebibliography}
\end{document}